\documentclass[a4paper, reqno, 12pt]{amsart}
\pdfoutput=1
\usepackage[utf8]{inputenc}
\usepackage[T1]{fontenc}
\usepackage{amssymb,mathtools}
\usepackage{mathrsfs}
\usepackage{amstext}
\usepackage{amsfonts}
\usepackage{latexsym}
\usepackage{fancyref}
\usepackage[final]{pdfpages}
\usepackage{fancyhdr}
\usepackage{xr} 
\usepackage{textcomp}
\usepackage{subfiles}
\usepackage{dsfont}
\usepackage{tabularx} 
\usepackage{setspace}
\usepackage{lmodern}
\usepackage{longtable}
\usepackage{geometry}
\usepackage[all]{xy}
\usepackage{microtype}
\usepackage{afterpage}
\usepackage{soul}
\usepackage{mathbbol}

\usepackage{enumitem}
\setlist[enumerate,1]{label=\textup{(\arabic*)}}

\usepackage[lite]{amsrefs}
\newcommand*{\MRref}[2]{ \href{http://www.ams.org/mathscinet-getitem?mr=#1}{MR \textbf{#1}}}

\usepackage{xcolor}

\newcommand*{\arxiv}[1]{\href{http://www.arxiv.org/abs/#1}{arXiv: #1}}

\renewcommand*{\PrintDOI}[1]{\href{http://dx.doi.org/\detokenize{#1}}{doi: \detokenize{#1}}}

\BibSpec{book}{%
  +{}  {\PrintPrimary}                {transition}
  +{,} { \textit}                     {title}
  +{.} { }                            {part}
  +{:} { \textit}                     {subtitle}
  +{,} { \PrintEdition}               {edition}
  +{}  { \PrintEditorsB}              {editor}
  +{,} { \PrintTranslatorsC}          {translator}
  +{,} { \PrintContributions}         {contribution}
  +{,} { }                            {series}
  +{,} { \voltext}                    {volume}
  +{,} { }                            {publisher}
  +{,} { }                            {organization}
  +{,} { }                            {address}
  +{,} { \PrintDateB}                 {date}
  +{,} { }                            {status}
  +{}  { \parenthesize}               {language}
  +{}  { \PrintTranslation}           {translation}
  +{;} { \PrintReprint}               {reprint}
  +{.} { }                            {note}
  +{.} {}                             {transition}
  +{} { \PrintDOI}                   {doi}
  +{} { available at \url}            {eprint}
  +{}  {\SentenceSpace \PrintReviews} {review}
}

\usepackage[pdftitle={Nuclearity for partial crossed products by exact discrete groups},
pdfauthor={},
pdfsubject={Mathematics}
]{hyperref}

\numberwithin{equation}{section}

\theoremstyle{plain}
\newtheorem*{theorem*}{Theorem}
\newtheorem*{cor*}{Corollary}
\newtheorem*{prop*}{Proposition}
\newtheorem{thm}[subsection]{Theorem}
\newtheorem{cor}[subsection]{Corollary}
\newtheorem{lem}[subsection]{Lemma}
\newtheorem{prop}[subsection]{Proposition}

\theoremstyle{definition}
\newtheorem{defn}[subsection]{Definition}

\theoremstyle{remark}
\newtheorem{rem}[subsection]{Remark}

\newcommand{\NN}{\mathbb{N}}
\newcommand{\TT}{\mathbb{T}}

\newcommand{\CC}{\mathbb{C}}

\newcommand{\OO}{\mathcal{O}}

\newcommand*{\nb}{\nobreakdash}
\newcommand*{\Star}{\(^*\)\nobreakdash-}
\newcommand{\Cst}{\mathrm{C}^*}
\newcommand{\Wst}{\mathrm{W}^*}
\newcommand{\wst}{\mathrm{w}^*}
\newcommand*{\Bound}{\mathbb B}
\newcommand{\Hilm}[1][E]{\mathcal{#1}}
\newcommand{\id}{\mathrm{id}}
\newcommand{\op}{\mathrm{op}}
\newcommand{\mbb}{\mathbb{M}_{\mathfrak{t}}(\mathcal{B})}
\newcommand{\exopp}{(\ell^\infty(G)\otimes(A\otimes_{\mathrm{max}}A^{\mathrm{op}}))\rtimes_{\tau\otimes(\alpha\otimes\alpha^{\mathrm{op}})}G} 
\newcommand{\opp}{(A\otimes_{\mathrm{max}}A^{\mathrm{op}})\rtimes_{\alpha\otimes\alpha^{\mathrm{op}}}G} 
\newcommand{\actex}{(\{\ell^\infty(G)\otimes(A_g\otimes_{\mathrm{max}}A_g^{\mathrm{op}})\}_{g\in G},\{\tau_g\otimes(\alpha_g\otimes\alpha_g^{\mathrm{op}})\}_{g\in G})} 

\newcommand{\A}{\mathcal{A}}
\newcommand{\B}{\mathcal{B}}
\newcommand*{\red}{\mathrm r}
\newcommand*{\cstar}{\texorpdfstring{\(\Cst\)\nobreakdash-\hspace{0pt}}{*-}}

\renewcommand{\min}{\mathrm{min}}
\renewcommand{\max}{\mathrm{max}}

\newcommand*{\congto}{\xrightarrow\sim}

\newcommand*{\into}{\hookrightarrow}

\DeclarePairedDelimiterX{\braket}[2]{\langle}{\rangle}{#1\,\delimsize\vert\,\mathopen{}#2}
\DeclarePairedDelimiterX{\BRAKET}[2]{\langle}{\rangle}{\!\delimsize\langle#1\,\delimsize\vert\,\mathopen{}#2\rangle\delimsize\!}

\DeclarePairedDelimiterX{\setgiven}[2]{\{}{\}}{#1\,{:}\,\mathopen{}#2}

\author{Alcides Buss}

\address{Departamento de Matem\'atica, Universidade Federal de Santa Catarina, 88040-900 Florian\'opolis SC, Brazil}
\email{alcides.buss@ufsc.br}

\author{Dami{\'a}n Ferraro}
\address{Departamento de Matem\'atica y Estad\'istica del Litoral, Universidad de la Rep\'ublica, 50000 Salto, Uruguay}

\email{dferraro@unorte.edu.uy}

\author{Camila F. Sehnem}

\address{School of Mathematics and Statistics, Victoria University of Wellington, P.O. Box 600, Wellington 6140, New Zealand.}

\email{camila.sehnem@vuw.ac.nz}



\title{Nuclearity for partial crossed products by exact discrete groups}

\begin{document}

\begin{abstract} We show that the partial crossed product of a commutative $\Cst$\nb-algebra by an exact discrete group is nuclear whenever the full and reduced partial crossed products coincide. This generalises a result by Matsumura in the context of global actions. In general, we prove that a partial action of an exact discrete group on a $\Cst$\nb-algebra~$A$ has Exel's approximation property if and only if the full and reduced partial crossed products by the diagonal partial action on~$A\otimes_{\max} A^\mathrm{op}$ coincide. We apply our results to establish nuclearity in terms of a weak containment property in the case of semigroup $\Cst$\nb-algebras and $\Cst$\nb-algebras associated to separated graphs.
\end{abstract}


\maketitle

\section{Introduction}

Anantharaman-Delaroche introduced in \cite{Anantharaman-Delaroche1987} important notions of amenability and approximation properties for actions of discrete groups on von Neumann algebras and $\Cst$\nb-algebras. Namely, an action $(M,\gamma)$ of a discrete group~$G$ on a von Neumann algebra~$M$ is amenable in the sense of Anantharaman-Delaroche if there exists a conditional expectation $\ell^\infty(G,M)\to M$ that is equivariant with respect to the diagonal action $\tau\otimes\gamma$ of~$G$ on~$\ell^\infty(G,M)$, where $\tau\colon G\to\ell^\infty(G)$ is the left translation action. She gave equivalent characterisations for amenability in~\cite[Th{\'e}or{\`e}me~3.3]{Anantharaman-Delaroche1987}, involving approximation properties and positive type functions. An action on a $\Cst$\nb-algebra~$A$ is then said to be amenable if its unique normal extension to the enveloping von Neumann algebra~$A''$ is amenable. With this definition, she proved that the crossed product of a nuclear $\Cst$\nb-algebra by a discrete group is nuclear if and only if the action is amenable.

Motivated by the work of Anantharaman-Delaroche, Matsumura used the Haagerup standard form of a von Neumann algebra~\cite{Haagerup} to show that an action of an exact discrete group on a unital commutative $\Cst$\nb-algebra is amenable whenever the full and reduced crossed products coincide~\cite{Matsumura}. The nonunital case is due to Buss, Echterhoff and Willett, who also generalised Matsumura's result to cover actions of exact locally compact groups \cite{Buss-Echter-Willett, buss2020amenability}. As a consequence, the crossed product associated to an action $(A,\alpha)$ on a commutative $\Cst$\nb-algebra is nuclear whenever $G$ is exact and $A\rtimes_\alpha G=A\rtimes_{\alpha,\red}G$. This latter property is what is sometimes called \emph{weak containment}. For an action of an exact discrete group on a nuclear $\Cst$\nb-algebra, nuclearity of the associated crossed product is equivalent to the weak containment property of the diagonal action on~$A\otimes A^\mathrm{op}$ (see \cite{Matsumura} and \cite[Theorem~5.8]{Buss-Echter-Willett}).

It is well known that nuclearity of a reduced partial crossed product $A\rtimes_{\alpha,\red}G$ implies that the left regular representation implements an isomorphism $A\rtimes_\alpha G\cong A\rtimes_{\alpha, r} G$ between the full and reduced crossed products. An analogous fact is even true for Fell bundles over discrete groups \cite[Theorem~25.11]{Exel:Partial_dynamical}. Our main purpose in this paper is to address the converse implication. Abadie, Buss and Ferraro have recently introduced in~\cite{2019arXiv190703803A} the concept of a partial action on a von Neumann algebra, in order to study amenability and several notions of approximation properties for partial actions and Fell bundles. They observed that a partial action $\alpha=(\{A_g\}_{g\in G},\{\alpha_g\}_{g\in G})$ on a $\Cst$\nb-algebra always induces a $\Wst$\nb-partial action 
$\alpha''=(\{A''_g\}_{g\in G}, \{\alpha''_g\}_{g\in G})$ on~$A''$ in a canonical way. We use this bidual partial action to represent the partial crossed product~$A\rtimes_\alpha G$ on a Hilbert space so that $A''$ is in standard form, following the ideas from the original work of Matsumura~\cite{Matsumura}, also employed in~\cite{Buss-Echter-Willett, buss2020amenability}. Our main theorem is the following:

\begin{theorem*}  Let $\alpha=(\{A_g\}_{g\in G},\{\alpha_g\}_{g\in G})$ be a partial action of an exact discrete group~$G$ on a commutative $\Cst$\nb-algebra~$A$. Suppose that $A\rtimes_{\alpha}G=A\rtimes_{\alpha,\red}G$. Then the partial crossed product $A\rtimes_{\alpha}G$ is nuclear, or equivalently, the partial action $\alpha$ has Exel's approximation property.
\end{theorem*}

In order to adapt the ideas of Matsumura to the context of partial actions, we prove a technical result concerning completely positive maps between graded $\Cst$\nb-algebras. Roughly speaking, Proposition~\ref{prop:ccp-reduced} tells us that a grading-preserving completely positive map from the full cross-sectional $\Cst$\nb-algebra of a Fell bundle over~$G$ into a $G$-graded $\Cst$\nb-algebra induces a completely positive  map between the \emph{reduced} cross-sectional $\Cst$\nb-algebras of the underlying bundles. This becomes a crucial tool in our proofs. The corresponding result holds for the \emph{full} cross-sectional $\Cst$\nb-algebras as well, but our proof involves a theory of completely positive maps between Fell bundles and a Stinespring's Dilation Theorem in this setting. We have chosen to include these results in the appendix, as they might be of independent interest.

We leave open the question whether the weak containment property for a Fell bundle~$\Hilm[B]=(B_g)_{g\in G}$ over an exact discrete group with commutative unit fibre algebra implies nuclearity of $\Cst_\red(\Hilm[B])$. A recent example by Buss, Echterhoff and
Willett shows that this is not true in general for nondiscrete, exact locally compact groups (see \cite[Proposition~5.26]{buss2020amenability} and \cite[Example~5.27]{buss2020amenability}). Their example consists of an exact, nonamenable locally compact group~$G$ and a $2$\nb-cocycle $w\colon G\times G\to \TT$ so that the full and reduced twisted group algebras coincide. Proposition 5.29 and Lemma~5.30 of \cite{buss2020amenability} suggest that constructing a similar example with twisted group algebras in the discrete case might be a difficult task. 

In the context of partial actions of exact discrete groups on not necessarily commutative $\Cst$\nb-algebras, we obtain a characterisation of the approximation property of a partial action in terms of the weak containment property of the diagonal partial action on~$A\otimes_\max A^{\mathrm{op}}$, in a similar fashion as in \cite{Matsumura} and  \cite{Buss-Echter-Willett, buss2020amenability}. To do this, we apply recent results on approximation properties and amenability for Fell bundles and partial actions due to Abadie, Buss and Ferraro~\cite{2019arXiv190703803A}. By \cite[Corollary~4.23]{buss2020amenability} (see also \cite[Proposition~7.2]{2019arXiv190703803A}), the reduced cross-sectional $\Cst$\nb-algebra $\Cst_\red(\Hilm[B])$ of a Fell bundle is nuclear if and only if its unit fibre \cstar{}algebra $B_e$ is nuclear and $\Hilm[B]$ has the approximation property. So in order to conclude that  a partial action~$\alpha$ has the approximation property,  by Corollary~6.10 and Theorem~6.12 of~\cite{2019arXiv190703803A} it suffices to show that the partial crossed product of the center of $A''$ by the bidual partial action is nuclear. We prove this in Corollary~\ref{cor:Z-weak-containment} assuming that $G$ is exact and the diagonal partial action satisfies the weak containment property. We then obtain our second main result.

\begin{theorem*} Let~$G$ be an exact discrete group and let $\alpha=(\{A_g\}_{g\in G},\{\alpha_g\}_{g\in G})$ be a partial action of~$G$ on a $\Cst$\nb-algebra~$A$. Then the following are equivalent:
\begin{enumerate}
\item[\rm{(i)}] $A\rtimes_{\alpha,\red}G$ has Exel's approximation property;

\item[\rm{(ii)}] $(A\otimes_\max B)\rtimes_{\alpha\otimes\beta}G=(A\otimes_\max B)\rtimes_{\alpha\otimes\beta, \red}G$ for every partial action $(\{B_g\}_{g\in G},\allowbreak\{\beta_g\}_{g\in G})$ of~$G$ on a $\Cst$\nb-algebra~$B$;

\item[\rm{(iii)}] $(A\otimes_\max A^{\mathrm{op}})\rtimes_{\alpha\otimes\alpha^{\mathrm{op}}}G=(A\otimes_\max A^{\mathrm{op}})\rtimes_{\alpha\otimes\alpha^{\mathrm{op}}, \red}G$.
\end{enumerate}
\end{theorem*}

In particular, nuclearity for the reduced partial crossed product of a nuclear $\Cst$\nb-algebra~$A$ by an exact discrete group corresponds to the weak containment property for the diagonal partial action on~$A\otimes A^\mathrm{op}$. An analogue of the above result in the setting of global actions follows from Theorem~5.16 and Corollary~5.17 of~\cite{buss2020amenability}.

The first application of our main theorem concerns semigroup $\Cst$\nb-algebras. A natural choice of a concrete $\Cst$\nb-algebra~$\Cst_{\lambda}(P)$ to attach to a left-cancellative semigroup~$P$ is the $\Cst$\nb-algebra generated by the canonical representation of~$P$ by isometries on~$\ell^2(P)$. Nica introduced a universal $\Cst$\nb-algebra associated to the positive cone of a quasi-lattice order in~\cite{Nica:Wiener--hopf_operators}. Li constructed full semigroup $\Cst$\nb-algebras out of an arbitrary left-cancellative semigroup~$P$ with the help of the family of constructible right ideals of~$P$~\cite{Li:Semigroup_amenability}. In this paper we consider the semigroup $\Cst$\nb-algebra~$\Cst_s(P)$ of a submonoid of a group. Nuclearity of the semigroup $\Cst$\nb-algebra~$\Cst_{\lambda}(P)$ implies $\Cst_s(P)\cong\Cst_{\lambda}(P)$ via the left regular representation, at least when~$P$ satisfies independence \cite[Theorem~5.6.44]{CLEY}. We obtain the converse implication if~$P$ embeds into an exact discrete group~$G$, using a realisation of~$\Cst_{\lambda}(P)$ as a partial crossed product of a commutative $\Cst$\nb-algebra by~$G$ due to Li~\cite[Theorem~5.6.41]{CLEY}. Precisely, we have from Theorem~\ref{thm:maintheorem}:

\begin{cor*} Let $P$ be a monoid that embeds into an exact discrete group~$G$. Suppose that the left regular representation $\lambda^+\colon \Cst_s(P)\to\Cst_{\lambda}(P)$ is an isomorphism. Then $\Cst_{\lambda}(P)$ is nuclear. 
\end{cor*}

With the above application, we give a positive answer to a question posed by Anantharaman-Delaroche~\cite[Remark~4.17]{anantharamandelaroche2016remarks} and add to the existing results in the literature on nuclearity of semigroup $\Cst$\nb-algebras (see \cite[Theorem~5.6.41]{CLEY} and also \cite{ABCD, LOS, LI2013626}). We also include a second application of Theorem~\ref{thm:maintheorem} in Corollary~\ref{cor:separated-graph} concerning $\Cst$\nb-algebras associated to separated graphs.

\section{Partial actions and some constructions}

In this section,  we recall some constructions related to partial actions of discrete groups on $\Cst$\nb-algebras. Exel introduced the partial crossed product of a $\Cst$\nb-algebra by the integers in \cite{Exel:Circle_actions}. The construction of a partial crossed product by an arbitrary discrete group is due to McClanahan \cite{Mcclanahan1995KTheoryFP}. Abadie, Buss and Ferraro introduced and used partial actions in the $\Wst$\nb-context in~\cite{2019arXiv190703803A} to study approximation properties and amenability of $\Cst$\nb-algebraic partial actions and Fell bundles.

\subsection{Partial actions in the \texorpdfstring{$\mathrm{C}^*$}{C*}-context}

\begin{defn}[\cite{Exel:Partial_dynamical}*{Definition~11.4}]\label{defn: partial-action} 
A \emph{partial action} of a discrete group~$G$ on a $\Cst$\nb-algebra~$A$ is a pair $\alpha=(\{A_g\}_{g\in G},\{\alpha_g\}_{g\in G})$, where  $\{A_g\}_{g\in G}$ is a family of closed two-sided ideals of~$A$ and $\alpha_g\colon A_{g^{-1}}\to A_g$ is a \Star isomorphism for each~$g\in G$, such that for all~$g,h\in G$

\begin{enumerate}
\item[\rm{(i)}] $A_e=A$ and $\alpha_e$ is the identity on~$A$ ($e$ being the unit element of~$G$);

\item[\rm{(ii)}] $\alpha_g(A_{g^{-1}}\cap A_h)\subseteq A_{gh}$;

\item[\rm{(iii)}] $\alpha_g\circ \alpha_h=\alpha_{gh}$ on~$A_{h^{-1}}\cap A_{(gh)^{-1}}$.
\end{enumerate}
\end{defn}

We will briefly recall the construction of a partial crossed product using cross-sectional $\Cst$\nb-algebras of Fell bundles, see \cite[Proposition~16.28]{Exel:Partial_dynamical}. The precise definition of a Fell bundle over a discrete group can be found in \cite[Definition~16.1]{Exel:Partial_dynamical}; the cross-sectional $\Cst$\nb-algebra of a Fell bundle is defined in \cite[Definition~16.25]{Exel:Partial_dynamical}, while its reduced counterpart is defined in \cite[Definition~17.6]{Exel:Partial_dynamical}. Since we will be working with Fell bundles over discrete groups, we will view a Fell bundle as a subset of both its full and reduced cross-sectional $\Cst$\nb-algebras so that the Banach space structure of the fibres and the multiplication and involution operations of the bundle are the ones inherited from these $\Cst$\nb-algebras.

Given a partial action $\alpha=(\{A_g\}_{g\in G},\{\alpha_g\}_{g\in G})$, we build a Fell bundle $\Hilm[B]_\alpha=(B_{\alpha_g})_{g\in G}$ over~$G$ as follows. We set $B_{\alpha_g}\coloneqq A_g$ as a complex Banach space. For $a\in A_g$, we write~$a\delta_g$ for the corresponding element of~$B_{\alpha_g}$. The multiplication map is given by $$B_{\alpha_g}\times B_{\alpha_h}\to B_{\alpha_{gh}},\ (a\delta_g,b\delta_h)\mapsto (a\delta_g)\cdot (b\delta_h)\coloneqq \alpha_g(\alpha_{g^{-1}}(a)b)\delta_{gh},$$ for all~$g,h\in G$. This is well defined by item (ii) of Definition~\ref{defn: partial-action}. After spending a little effort, one can prove that the resulting multiplication operation on~$\Hilm[B]_{\alpha}$ is associative. For each $g\in G$, the involution $^*\colon B_{\alpha_g}\to B_{\alpha_{g^{-1}}}$ is given by $$(a\delta_g)^*\coloneqq \alpha_{g^{-1}}(a^*)\delta_{g^{-1}}, \qquad a\in A_g.$$ With these operations, 
$\Hilm[B]_\alpha=(B_{\alpha_g})_{g\in G}$ is a Fell bundle whose unit fibre algebra is~$A$. We refer to \cite[Proposition~16.6]{Exel:Partial_dynamical} for further details.

\begin{defn}[see \cite{Exel:Partial_dynamical}*{Proposition~16.28 and Definition~17.10}] The Fell bundle $\Hilm[B]_{\alpha}=(B_{\alpha_g})_{g\in G}$ constructed above is called the \emph{semidirect product bundle} of~$\alpha=(\{A_g\}_{g\in G},\{\alpha_g\}_{g\in G})$. The \emph{partial crossed product} of $A$ by $G$ under $\alpha=(\{A_g\}_{g\in G},\{\alpha_g\}_{g\in G})$, denoted by $A\rtimes_{\alpha}G$, is the (full) cross-sectional $\Cst$\nb-algebra of~$\Hilm[B]_{\alpha}=(B_{\alpha_g})_{g\in G}$. The \emph{reduced partial crossed product} $A\rtimes_{\alpha, \red}G$ is then defined to be the reduced cross-sectional $\Cst$\nb-algebra of~$\Hilm[B]_{\alpha}$.
\end{defn}

We will simply write $A\rtimes_{\alpha}G=A\rtimes_{\alpha, \red}G$ to say that the full and reduced partial crossed products are isomorphic via the left regular representation $\Lambda\colon\Cst(\Hilm[B]_{\alpha})\to \Cst_{\red}(\Hilm[B]_{\alpha})$.

\subsection{Covariant representations} Nondegenerate representations of the partial crossed product $A\rtimes_\alpha G$ on a Hilbert space $\Hilm[H]$ are parametrized by certain representations of $(\{A_g\}_{g\in G}, \{\alpha_g\}_{g\in G})$ on~$\Hilm[H]$. Recall that a \Star partial representation of $G$ on a unital $\Cst$\nb-algebra~$B$ is a map $v\colon G\to B$ from $G$ to the set of partial isometries in~$B$ with~$v_e=1$ and such that the set of partial isometries $\{v_g\mid g\in G\}$ satisfies the relations $$v_{g}^*=v_{g^{-1}}\qquad\text{ and }\qquad v_{g}v_hv_{h^{-1}}=v_{gh}v_{h^{-1}},$$ for all~$g, h\in G$.

\begin{defn} Given a Hilbert space $\Hilm[H]$ and a partial action $\alpha=(\{A_g\}_{g\in G}, \{\alpha_g\}_{g\in G})$ of~$G$ on a $\Cst$\nb-algebra~$A$, a \emph{covariant representation} of $\alpha$ in~$\Bound(\Hilm[H])$ is a pair $(\pi, v)$, where $v\colon G\to \Bound(\Hilm[H])$ is a \Star partial representation and $\pi\colon A\to \Bound(\Hilm[H])$ is a \Star homomorphism, such that for all $g\in G$ and $a\in A_{g^{-1}}$, $$v_g\pi(a)v_{g^{-1}}=\pi(\alpha_g(a)).$$
\end{defn}

\begin{prop}[\cite{Exel:Partial_dynamical}*{Proposition 13.1}] Let $(\pi, v)$ be a covariant representation of $\alpha=(\{A_g\}_{g\in G}, \{\alpha_g\}_{g\in G})$ in~$\Bound(\Hilm[H])$. Then there is a unique \Star representation $\pi\times v\colon A\rtimes_{\alpha}G\to \Bound(\Hilm[H])$ such that for all $g\in G$ and $a\in A_g$, $$(\pi\times v)(a\delta_g)=\pi(a)v_g.$$
\end{prop}

The \Star representation $\pi\times v$ is called the \emph{integrated form} of the covariant pair $(\pi, v)$. It follows from \cite[Theorem~13.2]{Exel:Partial_dynamical} that the map $(\pi, v)\mapsto \pi\times v$ is a one-to-one correspondence between covariant representations of $\alpha=(\{A_g\}_{g\in G}, \{\alpha_g\}_{g\in G})$ in~$\Bound(\Hilm[H])$ such that $v_gv_{g^{-1}}$ is the orthogonal projection onto $\pi(A_g)\Hilm[H]=\overline{\operatorname{span}}\{\pi(a)\xi\mid a\in A_g, \xi\in\Hilm[H]\}$ and nondegenerate \Star representations of the partial crossed product $A\rtimes_{\alpha}G$ on~$\Hilm[H]$.

\subsection{The diagonal partial action}

Let $\alpha=(\{A_g\}_{g\in G},\{\alpha_g\}_{g\in G})$ and $\beta=(\{B_g\}_{g\in G},\linebreak\{\beta_g\}_{g\in G})$ be partial actions of~$G$ on $\Cst$\nb-algebras~$A$ and $B$, respectively. Then $$\alpha\otimes\beta=(\{A_g\otimes_{\mu} B_g\}_{g\in G}, \{\alpha_g\otimes\beta_g\}_{g\in G})$$ is a partial action of~$G$ on~$A\otimes_{\mu} B$, where $\mu$ denotes either the maximal or the minimal tensor product. We will refer to $(\{A_g\otimes_{\mu} B_g\}_{g\in G}, \{\alpha_g\otimes\beta_g\}_{g\in G})$ as the \emph{diagonal partial action} obtained from $\alpha$ and $\beta$.

A particular example of a diagonal partial action will be important in this paper. Let $A$ be a $\Cst$\nb-algebra. The \emph{opposite algebra} of~$A$, denoted by $A^{\mathrm{op}}$, is the $\Cst$\nb-algebra with the same Banach space structure and involution operation as~$A$, but opposite multiplication. That is, for $a$ and $b$ in $ A^{\mathrm{op}}$, we have $a\cdot b\coloneqq ba$. A partial action $(\{A_g\}_{g\in G},\{\alpha_g\}_{g\in G})$ on~$A$ canonically induces a partial action $(\{A_g^{\mathrm{op}}\}_{g\in G}, \{\alpha^{\mathrm{op}}_g\}_{g\in G})$ on its opposite algebra. Here the isomorphism $\alpha_g^{\mathrm{op}}$ is simply $\alpha_g$ viewed as a map from $A_{g^{-1}}^{\mathrm{op}}$ to $A_g^{\mathrm{op}}$. Nuclearity and amenability for a group action on a $\Cst$\nb-algebra are closely related to the weak containment property for the diagonal action on~$A\otimes_{\mathrm{max}}A^{\mathrm{op}}$. See \cite[Theorem~1.1]{Matsumura} and \cite[Theorem~5.16]{buss2020amenability}. We will establish a similar connection for partial actions of exact discrete groups in Theorem~\ref{thm:noncommutative}.

\subsection{Bidual partial actions}

A \emph{$\Wst$\nb-partial action} of~$G$ on a $\Wst$\nb-algebra~$M$ is a partial action $\gamma=(\{M_g\}_{g\in G}, \{\gamma_g\}_{g\in G})$ of $G$ on $M$ as in Definition~\ref{defn: partial-action} (regarding $M$ as a $\Cst$\nb-algebra) with the additional requirement that each~$M_g$ be a $\Wst$\nb-ideal of $M.$
This implies that each $\gamma_g$ is a $\Wst$\nb-isomorphism because any \Star isomorphism between $\Wst$\nb-algebras is a $\Wst$\nb-isomorphism, that is, a normal \Star isomorphism.

By \cite[]{2019arXiv190703803A}, every partial action $\alpha=(\{A_g\}_{g\in G}, \{\alpha_g\}_{g\in G})$ induces a $\Wst$\nb-partial action on the double dual (enveloping) $\Wst$\nb-algebra $A''$ of~$A$. The collection of $\Wst$\nb-ideals of~$A''$ is given by the family of enveloping $\Wst$\nb-algebras $\{A_g''\}_{g\in G}$, where each $A_g''$ is regarded as a $\Wst$\nb-ideal of~$A''$ through the natural inclusion. Each \Star isomorphism $\alpha_g\colon A_{g^{-1}}\to A_g$ has a unique $\wst$\nb-continuous extension $\alpha_g''\colon A''_{g^{-1}}\to A_g''$, which is itself a $\Wst$\nb-isomorphism with inverse~$\alpha_{g^{-1}}''$. We will call $\alpha''=(\{A''_g\}_{g\in G},\{\alpha_g''\}_{g\in G})$ the \emph{bidual partial action} of~$\alpha=(\{A_g\}_{g\in G}, \{\alpha_g\}_{g\in G})$.

\begin{defn} A \emph{covariant $\Wst$\nb-representation} of a $\Wst$\nb-partial action $\gamma=(\{M_g\}_{g\in G}, \allowbreak \{\gamma_g\}_{g\in G})$ on a Hilbert space $\Hilm[H]$ is a covariant representation~$(\pi, v)$ of $\gamma$ in~$\Bound(\Hilm[H])$ such that $\pi$ is $\wst$\nb-continuous.
\end{defn}

\section{Haagerup standard form}

In this section, we recall some important aspects from the work of Haagerup~\cite{Haagerup} on the standard form of a von Neumann algebra. Matsumura used this presentation of a von Neumann algebra in~\cite{Matsumura} to establish sufficient conditions for nuclearity of a crossed product by an exact discrete group, motivated by the work of Anantharaman-Delaroche \cite{Anantharaman-Delaroche1987} on amenability for actions on von Neumann algebras and $\Cst$\nb-algebras. Buss, Echterhoff and Willett followed these ideas to generalise the main results of Matsumura in~\cite{Matsumura} to exact locally compact groups that are not necessarily discrete \cite{Buss-Echter-Willett, buss2020amenability}, and also to study and connect notions of amenability for group actions. In order to bring these ideas to the setting of partial actions, we will build for a $\Wst$\nb-partial action $(\{M_g\}_{g\in G}, \{\gamma_g\}_{g\in G})$ on a $\Wst$\nb-algebra $M$ a \Star partial representation of~$G$ on a standard form of~$M$, implementing the family of isomorphisms $\{\gamma_{g}\colon M_{g^{-1}}\to M_g\}_{g\in G}$.  See Proposition~\ref{prop:cov-standard} for further details.

We begin by stating a few important results and concepts of~\cite{Haagerup} that will be needed in the sequel.

\begin{thm}[\cite{Haagerup}*{Theorem 1.6}]\label{thm:standardform} Any von Neumann algebra is isomorphic to a von Neumann algebra~$M$ on a Hilbert space~$\Hilm[H]$ such that there exist a conjugate linear isometric involution $J\colon \Hilm[H]\to \Hilm[H]$ and a selfdual cone~$P$ in~$\Hilm[H]$ with the following properties:
\begin{enumerate}
\item[\rm{(1)}] $JMJ=M'$,
\item[\rm{(2)}] $JcJ=c^*$ for all~$c$ in the center $\mathrm{Z}(M)$,
\item[\rm{(3)}] $J\xi=\xi$ for all~$\xi\in P$,
\item[\rm{(4)}] $aa^t(P)\subseteq P$ for all $a\in M$, where $a^t=JaJ$.
\end{enumerate}
\end{thm}

A quadruple $(M, \Hilm[H], J,P)$ satisfying conditions (1)--(4) is called a \emph{Haagerup standard form} of~$M$.

Let~$q$ be a projection of the form~$pp'$, where $p\in M$ and $p'\in M'$ are projections. Set~$qMq=\{qmq\mid m\in M\}$ viewed as a set of operators on $q(\Hilm[H])$. Then $qMq$ is a von Neumann algebra  by \cite[Lemma~2.4]{Haagerup}. The following lemma will be important to construct the implementation of a $\Wst$\nb-partial action on a von Neumann algebra in standard form. See Corollary 2.5 and Lemma~2.6 of \cite{Haagerup}.

\begin{lem}\label{lem:restriction} Let $(M,\Hilm[H], J,P)$ be a standard form. Let $p\in M$ be a projection. Set~$p^t=JpJ$ and let $q=pp^t$. Then $(qMq, q(\Hilm[H]), qJq,q(P))$ is a standard form. Moreover, the induction $pMp\to qMq$ is an isomorphism of~$pMp$ onto~$qMq$.
\end{lem}

The next theorem concerns uniqueness of a Haagerup standard form of a von Neumann algebra. This is \cite[Theorem~2.3]{Haagerup}.

\begin{thm}[Uniqueness of standard form]\label{thm:uniqueness-stand} Let $(M_1, \Hilm[H]_1, J_1, P_1)$ and $(M_2, \Hilm[H]_2, J_2, P_2)$ be standard forms of von Neumann algebras~$M_1$ and $M_2$. Let $\phi\colon M_1\to M_2$ be a \Star isomorphism. Then there is a unique unitary $u\colon\Hilm[H]_1\to\Hilm[H]_2$ such that
\begin{enumerate}
\item[\rm{(1)}] $\phi(m)=umu^*$;

\item[\rm{(2)}] $J_2=uJ_1u^*$;

\item[\rm{(3)}] $P_2=u(P_1)$.

\end{enumerate}
\end{thm}

 Recall that a $\Wst$\nb-enveloping action $(N,\tilde{\gamma})$ of a $\Wst$\nb-partial action $\gamma=(\{M_g\}_{g\in G}, \allowbreak\{\gamma_g\}_{g\in G})$ consists of a $\Wst$\nb-global action $\tilde{\gamma}\colon G\to\mathrm{Aut}(N)$, and an embedding $ M\hookrightarrow N$ of~$M$ as a $\Wst$\nb-ideal of~$N$ in a way that $\gamma=(\{M_g\}_{g\in G}, \{\gamma_g\}_{g\in G})$ is isomorphic to the $\Wst$\nb-partial action given by the restriction of~$\tilde{\gamma}$ to~$M$, and such that the linear $\tilde{\gamma}$\nb-orbit of~$M$ is $\wst$\nb-dense in~$N$. A $\Wst$\nb-partial action always has a $\Wst$\nb-enveloping action by \cite[Proposition~2.7]{2019arXiv190703803A}, which is unique up to isomorphism. The notion of enveloping actions was originally introduced by Abadie \cite{Abadie} in the $\Cst$\nb-context. 

Using $\Wst$\nb-enveloping actions and restriction and uniqueness of a Haagerup standard form, we can prove the following. 

\begin{prop}\label{prop:cov-standard} Let $\gamma=(\{M_g\}_{g\in G}, \{\gamma_g\}_{g\in G})$ be a $\Wst$\nb-partial action on a $\Wst$\nb-algebra $M$ and let $(M, \Hilm[H], J, P)$ be a Haagerup standard form of~$M$. Let $\iota\colon M\to\Bound(\Hilm[H])$ be the inclusion. There exists a \Star partial representation $v\colon G\to\Bound(\Hilm[H])$ such that~$(\iota, v)$ is a covariant $\Wst$\nb-representation of~$ \gamma=(\{M_g\}_{g\in G}, \{\gamma_g\}_{g\in G})$. Moreover, the pair $(\iota^{\operatorname{op}}, v)$ is a covariant $\Wst$\nb-representation of $\gamma^{\mathrm{op}}=(\{M_g^{\mathrm{op}}\}_{g\in G}, \{\gamma^{\mathrm{op}}_g\}_{g\in G})$, where $\iota^{\mathrm{op}}\colon M^{\mathrm{op}}\to \Bound(\Hilm[H])$ is the $\Wst$\nb-isomorphism $m\mapsto Jm^*J$ from $M^{\mathrm{op}}$ onto~$\iota(M)'$.
\begin{proof} Let $(N,\tilde{\gamma})$ be the $\Wst$\nb-enveloping action of $\gamma$. Let $(N,\tilde{\Hilm[H]}, \tilde{J},\tilde P)$ be a Haagerup standard form of~$N$. We view~$N$ as a von Neumann subalgebra of~$\Bound(\tilde{\Hilm[H]})$ and so we introduce no special notation for the inclusion map $N\hookrightarrow \Bound(\tilde{\Hilm[H]})$. By \cite[Theorem~3.2]{Haagerup}, there exists a unique unitary representation $u\colon G\to\Bound(\tilde{\Hilm[H]})$ implementing the family of \Star automorphisms $\{\tilde{\gamma}_g\colon N\to N\}_{g\in G}$ and satisfying for all $g\in G$, $$\tilde J=u_g\tilde J u_{g^{-1}}\quad\text{ and } \quad u_g(\tilde P)=\tilde P.$$ Let $1_M\in M\subset N$ be the unit of~$M$. Then $1_M$ is a central projection in~$N$. Hence $(M, 1_M(\tilde{\Hilm[H]}),1_M\tilde J 1_M, 1_M(\tilde{P}))$ is a standard form by  Lemma~\ref{lem:restriction}. Since $1_M$ is a central projection in~$N$, the map $v'\colon G\to \Bound(1_M(\tilde{\Hilm[H]}))$ given by $v_g'\coloneqq 1_Mu_g1_M$ is a \Star partial representation (see \cite[Proposition~9.5]{Exel:Partial_dynamical}). This yields a covariant $\Wst$\nb-representation of~$\gamma$ when combined with the inclusion $M\hookrightarrow\Bound(1_{M}(\tilde{\Hilm[H]}))$. Because $\tilde{J}u_g=u_g\tilde{J}$ in~$\Bound(\tilde{\Hilm[H]})$ for every~$g\in G$, the map that sends $m\in M^{\mathrm{op}}$ to $1_M\tilde{J}m^*\tilde{J}1_M\in \Bound(1_M(\tilde{\Hilm[H]}))$ and the \Star partial representation~$v'\colon G\to\Bound(1_M(\tilde{\Hilm[H]}))$ give a covariant pair for $\gamma^{\mathrm{op}}$.

Now given an arbitrary standard form~$(M,\Hilm[H],J,P)$ of~$M$ with inclusion map~$\iota\colon M\to\Bound(\Hilm[H])$, let~$v\colon G\to\Bound(\Hilm[H])$ be the canonical \Star partial representation obtained from the \Star partial representation of~$G$ in $\Bound(1_M(\tilde{\Hilm[H]}))$ built above and the unitary implementing the \Star isomorphism $M\cong \iota(M)$ from Theorem~\ref{thm:uniqueness-stand}. Thus $(\iota, v)$ is a covariant $\Wst$\nb-representation of~$\gamma=(\{M_g\}_{g\in G}, \{\gamma_g\}_{g\in G})$ in~$\Bound(\Hilm[H])$ such that  $(\iota^{\operatorname{op}}, v)$ is a covariant $\Wst$\nb-representation of~$\gamma^{\mathrm{op}}$. This completes the proof of the proposition.
\end{proof}
\end{prop}

\begin{rem} It is also possible to prove Proposition~\ref{prop:cov-standard} by building directly a \Star partial representation of~$G$ in~$\Bound(\Hilm[H])$ implementing the family of \Star isomorphisms $\{\gamma_g\colon M_{g^{-1}}\to M_g\}_{g\in G}$, where $(M,\Hilm[H], J,P)$ is a Haagerup standard form of~$M$. The proof follows along the same lines as that of \cite[Theorem~3.2]{Haagerup}. Indeed, by Lemma~\ref{lem:restriction} and Theorem~\ref{thm:uniqueness-stand} there exists a unique unitary $$\bar{v}_g\colon 1_{g^{-1}}(\Hilm[H])\to 1_g(\Hilm[H])$$ implementing the \Star isomorphism~$\gamma_g\colon M_{g^{-1}}\to M_g$ and such that $$1_gJ1_g=\bar{v}_g1_{g^{-1}}J1_{g^{-1}}\bar{v}_g^* \quad\text{ and }\quad
\bar{v}_g(1_{g^{-1}}(P))=1_{g}(P).$$ Thus $\gamma_g(m)=\bar{v}_gm\bar{v}_g^*$ for all $m\in M_{g^{-1}}$. Let $v_g$ be the unique partial isometry on~$\Hilm[H]$ such that~$v_g= \bar{v}_g$ on $1_{g^{-1}}(\Hilm[H])$ and $v_g\equiv 0$ on the orthogonal complement $1_{g^{-1}}(\Hilm[H])^\perp$. In particular, $v_gv_{g^{-1}}=1_g$. After some applications of Lemma~\ref{lem:restriction} and of the uniqueness of the unitary in Theorem~\ref{thm:uniqueness-stand}, one can show that $g\mapsto v_g$ is a \Star partial representation of~$G$ in~$\Bound(\Hilm[H])$ satisfying the desired properties.
\end{rem}

\section{Nuclearity for partial crossed products}\label{sec:mainsection}

	In this section we generalise results due to Matsumura in \cite{Matsumura} on nuclearity for crossed products by exact discrete groups to partial crossed products. We show that the partial crossed product of a commutative $\Cst$\nb-algebra by an exact discrete group is nuclear whenever the left regular representation implements an isomorphism between the full and reduced partial crossed products. In general, we show that a partial action of an exact discrete group on a $\Cst$\nb-algebra~$A$ has the approximation property (see Definition~\ref{defn:AP}) if and only if the full and reduced partial crossed products associated to the diagonal partial action on~$A\otimes_\max A^\mathrm{op}$ coincide (see Theorem~\ref{thm:noncommutative} for details). In particular, the partial crossed product of a nuclear $\Cst$\nb-algebra by an exact discrete group is again nuclear if and only if the diagonal partial action $\alpha\otimes \alpha^{\mathrm{op}}$ satisfies the weak containment property. For ordinary (global) actions of discrete groups on (unital) $\Cst$\nb-algebras, these results were obtained  by Matsumura in \cite{Matsumura} under the nuclearity assumption on~$A$ in the noncommutative version. They have been recently extended to general locally compact group actions by Buss, Echterhoff and Willett in \cite{buss2020amenability}, without assuming nuclearity.

\subsection{Approximation properties} The following definition of approximation property for a Fell bundle is due to Exel~\cite[Definition~20.4]{Exel:Partial_dynamical}.

\begin{defn}\label{defn:AP} Let $\Hilm[B]=(B_g)_{g\in G}$ be a Fell bundle over a discrete group~$G$. We say that $\Hilm[B]$ has the \emph{approximation property} if
there exists a net $(\xi_i\colon G \to B_e)_{i\in I}$ of finitely supported functions with the following properties:
\begin{itemize}
\item[\rm{(i)}] $\sup_{\substack{i\in I}} \|\sum_{\substack{g\in G}}\xi_i(g)^*\xi_i(g)\|<\infty$;

\item[\rm{(ii)}] $\lim_{\substack{i}} \sum_{\substack{h\in G}}\xi_i(gh)^*a\xi_i(h)=a$, for all $g\in G$  and $a\in B_g$ .
\end{itemize}

A partial action $\alpha=(\{A_g\}_{g\in G},\{\alpha_g\}_{g\in G})$ has the approximation property if the associated Fell bundle $\Hilm[B]_{\alpha}$ has the approximation property.
\end{defn}

A weak variant of the approximation property for Fell bundles as well as a version of Anantharaman-Delaroche's amenability  
were defined in \cite[Definition~6.3]{2019arXiv190703803A} and \cite[Definition~5.20]{2019arXiv190703803A}, respectively.
But all this turns out to be equivalent to the approximation property of Exel by the recent results obtained in \cite[Theorem~4.22]{buss2020amenability}, see \cite[Theorem~6.12]{2019arXiv190703803A}. In particular, the nuclearity of the cross-sectional  \cstar{}algebra $\Cst(\B)$ implies the approximation property of $\B$ by \cite[Corollary~4.23]{buss2020amenability}. We shall use these results in what follows.

Given Fell bundles $\A=(A_g)_{g\in G}$ and $\B=(B_g)_{g\in G}$, we will be concerned with the Fell bundle over~$G$ obtained from the maximal tensor product of~$\A$ and $\B$. This is the restriction of the maximal tensor product of~$\A$ and $\B$ as in \cite{Abadie:Tensor}*{Proposition~3.10} to the diagonal subgroup $G\cong\{(g,g)\mid g\in G\}\leq G\times G$, and so its fibre at $g\in G$ is $A_g\otimes_\max B_g$. We will simply denote by $\A\otimes_\max \B$ the resulting Fell bundle over~$G$, but we observe that this notation is adopted in~\cite{Abadie:Tensor} for the Fell bundle over~$G\times G$ produced from~$\A$ and~$\B$. Similarly, we let $\A\otimes_\min \B$ denote the Fell bundle over~$G$ whose fibre at $g\in G$ is $A_g\otimes_\min B_g$. We shall need the following permanence property of Exel's approximation property with respect to tensor products of Fell bundles.

\begin{prop}\label{prop:nuclearity-diag-tensor} 
Let $\A=(A_g)_{g\in G}$ and $\B=(B_g)_{g\in G}$ be Fell bundles over a discrete group~$G$ with unit fibre algebras~$A$ and $B$, respectively. If $\A$ has the approximation property, 
then so do the maximal and minimal tensor products 
$\A\otimes_\max\B$ and $\A\otimes_\min\B$. In particular, $\Cst(\A\otimes_\mu \B)=\Cst_\red(\A\otimes_\mu \B)$ if $\mu$ denotes either the maximal or the minimal tensor product.  If, in addition, $A$ and~$B$ are nuclear, then so is $\Cst(\A\otimes\B)$.
\begin{proof} Since~$\A$ has the approximation property, there exists a net of finitely supported functions $(\xi_i \colon G\to A)_{i\in I}$ with $\sum_{\substack{g\in G}}\xi_i(g)^*\xi_i(g)\leq M$ for all~$i\in I$ for some constant $M>0$, satisfying for all~$g\in G$, $$\lim_{\substack{i}}\sum_{\substack{h\in G}}\xi_i(gh)^*a \xi_i(h))=a$$
for all $a\in A_g$. To see that $\A\otimes_\mu\B$ also has the approximation property, let $(u_{j})_{j\in J}$ be an approximate identity for~$B$. For $(i,j)\in I\times J$, set 
\begin{equation*}
\begin{aligned}\eta_{i,j}\colon G&\to A\otimes_\mu B\\
g&\mapsto \xi_i(g)\otimes u_{j}.
\end{aligned}
\end{equation*}
Then $(\eta_{i,j}\colon G\to  A\otimes_\mu B)_{(i,j)\in I\times J}$ is a net of finitely supported functions with $$\sup_{(i,j)\in I\times J}\|\sum_{\substack{g\in G}}\eta_{i,j}(g)^*\eta_{i,j}(g)\|\leq M.$$ Moreover, for~$g\in G$ and an elementary tensor $c=a\otimes b\in A_g\otimes B_g$, we have 
\begin{equation*}\begin{aligned}\lim_{\substack{(i,j)}} \sum_{\substack{h\in G}}\eta_{i,j}(gh)^*c \eta_{i,j}(h)&=\lim_{\substack{(i,j)}} \sum_{\substack{h\in G}}(\xi_i(gh)^*\otimes u_{j})(a\otimes b)(\xi_i(h)\otimes u_{j})\\&=\lim_{\substack{(i,j)}} \sum_{\substack{h\in G}}\xi_i(gh)^*a\xi_i(h) \otimes u_j bu_j\\&=c.
\end{aligned}
\end{equation*} By continuity, the same holds for all $c\in A_g\otimes_\mu B_g$ and this shows that $\A\otimes_\mu\B$ has the approximation property. In particular, $\Cst(\A\otimes_\mu \B)=\Cst_\red(\A\otimes_\mu \B)$ by \cite[Theorem~20.6]{Exel:Partial_dynamical}. If, in addition, $A$ and $B$ are nuclear, then nuclearity of $\Cst(\A\otimes\B)$ follows from \cite[Proposition~25.10]{Exel:Partial_dynamical}.
\end{proof}
\end{prop}

Let $\alpha$ and $\beta$ be partial actions on $\Cst$\nb-algebras $A$ and $B$, respectively. The tensor product $\B_\alpha\otimes_\mu\B_\beta$ of their semidirect product bundles is canonically isomorphic to the semidirect product bundle associated to the diagonal partial action on $A\otimes_\mu B$. Specifically, the map that sends an elementary tensor $a\delta_g\otimes b\delta_g\in \B_{\alpha_g}\otimes_\mu\B_{\beta_g}$ to $(a\otimes b)\delta_g\in \B_{\alpha_g\otimes\beta_g}$ induces an isomorphism $\varphi_g\colon  \B_{\alpha_g}\otimes_\mu\B_{\beta_g}\congto \B_{\alpha_g\otimes\beta_g}$ of Banach spaces, and the collection $\varphi=\{\varphi_g\}_{g\in G}$ produces an isomorphism $\B_\alpha\otimes_\mu\B_\beta\cong  \B_{\alpha\otimes\beta}$ of Fell bundles (see \cite[Definition~21.1]{Exel:Partial_dynamical}). This then immediately yields the following:

\begin{cor}\label{cor:nuclearity-diag} 
Let $\alpha=(\{A_g\}_{g\in G},\{\alpha_g\}_{g\in G})$ and $\beta=(\{B_g\}_{g\in G},\{\beta_g\}_{g\in G})$ be partial actions of a discrete group~$G$ on $\Cst$\nb-algebras~$A$ and $B$, respectively. If $\alpha$ has the approximation property, then so does the diagonal partial action $$(\{A_g\otimes_\mu B_g\}_{g\in G},\{\alpha_g\otimes \beta_g\}_{g\in G})$$ of $G$ on $A\otimes_\mu B$, where $\mu$ denotes either the maximal or the minimal tensor product. In particular, $$(A\otimes_\mu B)\rtimes_{\alpha\otimes\beta}G=(A\otimes_\mu B)\rtimes_{\alpha\otimes\beta, \red}G.$$ If, in addition, $A$ and $B$ are nuclear, then so is the partial crossed product $(A\otimes B)\rtimes_{\alpha\otimes\beta}G$.
\end{cor}

For us the main application of the above result will be the following corollary. It also explains the main role of the exactness assumption of the underlying group in our main results.
To state the result we need to recall an important characterisation of exactness for discrete groups. This is due to Ozawa~\cite{Ozawa}, motivated by the work of Guentner and Kaminker~\cite{GUENTNER2002411}.

\begin{thm} Let $G$ be a discrete group. The following are equivalent:
\begin{enumerate}
\item[\rm{(i)}] $G$ is exact; 
\item[\rm{(ii)}] $\ell^{\infty}(G)\rtimes_{\tau, \red}G$ is nuclear.
\item[\rm(iii)] The (left) translation action $\tau$ of $G$ on $\ell^\infty(G)$ is (strongly) amenable, or equivalently, it has the approximation property.
\end{enumerate}
\end{thm}

The left translation action $\tau$ used above is defined  by $$\tau_g(f)(s)=f(g^{-1}s)\qquad (s\in G)$$
for~$g\in G$ and $f\in \ell^\infty(G)$. The reduced crossed product $\ell^\infty(G)\rtimes_{\tau,\red}G$ is canonically isomorphic to the \emph{uniform Roe algebra} of~$G$,  the $\Cst$\nb-subalgebra of~$\Bound(\ell^{2}(G))$ generated by $\ell^{\infty}(G)$ and the left regular representation $\lambda\colon G\to\Bound(\ell^2(G))$ of~$G$.

\begin{cor}\label{cor:exact-nuclearity-diag}
If $G$ is exact, then for every partial action $(\{B_g\}_{g\in G},\{\beta_g\}_{g\in G})$  of $G$ on a \cstar{}algebra $B$, the diagonal partial action 
$$(\{\ell^\infty(G)\otimes B_g\}_{g\in G},\{\tau_g\otimes \beta_g\}_{g\in G})$$ 
of $G$ on $\ell^\infty(G)\otimes B$ has the approximation property. 
Hence 
$$(\ell^\infty(G)\otimes B)\rtimes_{\tau\otimes\beta}G=(\ell^\infty(G)\otimes B)\rtimes_{\tau\otimes\beta,\red}G.$$ 
If, in addition, $B$ is nuclear, then so is the partial crossed product 
$(\ell^\infty(G)\otimes B)\rtimes_{\tau\otimes\beta}G$.
\end{cor}

\subsection{Partial actions on commutative \texorpdfstring{$\mathrm{C}^*$}{C*}-algebras}\label{sec:partial-actions-commutative} 

Recall that a representation $\pi=\{\pi_g\}_{g\in G}$ of a Fell bundle in a $\Cst$\nb-algebra~$C$ is said to be \emph{faithful} if $\pi_e\colon B_e\to C$ is injective. In this case, each $\pi_g\colon B_g\to C$ is isometric. We begin by proving a technical result on certain grading-preserving completely positive maps. This will be an important tool in our proofs. In what follows we will shorten the notation and write cp (resp. ccp) for (contractive) completely positive maps.

\begin{prop}\label{prop:ccp-reduced} Let $\Hilm[B]=(B_g)_{g\in G}$ and $\Hilm[C]=(C_g)_{g\in G}$ be Fell bundles. Let $\pi=\{\pi_g\}_{g\in G}$ be a faithful representation of~$\Hilm[C]$ on a Hilbert space~$\Hilm[H]$. Suppose that $\phi\colon\Cst(\Hilm[B])\to\Bound(\Hilm[H])$ is a cp map such that $\phi(B_g)\subseteq \pi_g(C_g)$ for all~$g\in G$. Then there is a cp map between reduced cross-sectional $\Cst$\nb-algebras $\phi'\colon \Cst_r(\Hilm[B])\to\Cst_r(\Hilm[C])$ with $\|\phi'\|=\|\phi\|$ such that $\pi\circ \phi'|_{\Hilm[B]} = \phi.$
\begin{proof} As in the statement, let $\pi=\{\pi_g\}_{g\in G}$ be a faithful representation of~$\Hilm[C]$ on a Hilbert space~$\Hilm[H]$ and $\phi\colon\Cst(\Hilm[B])\to\Bound(\Hilm[H])$ a cp map satisfying $\phi(B_g)\subseteq\pi_g(C_g)$ for all~$g\in G$. By Stinespring's Dilation Theorem for cp maps (see \cite[Theorem~1]{Stinespring}), there exist a Hilbert space $\Hilm[K]$, a \Star homomorphism $\hat{\phi}\colon \Cst(\Hilm[B])\to\Bound(\Hilm[K])$ and an operator $V\colon \Hilm[H]\to\Hilm[K]$ such that $$\phi(b)=V^*\hat{\phi}(b)V$$ for all $b\in\Cst(\Hilm[B])$. Let us still denote by $\hat{\phi}$ the corresponding representation of $\Hilm[B]$ in~$\Bound(\Hilm[K])$ and let $\lambda\colon G\to \Bound(\ell^2(G))$ be the left regular representation. By Fell's absorption principle for Fell bundles (see \cite[Proposition~18.4]{Exel:Partial_dynamical}), the integrated forms of $\lambda\otimes\hat{\phi}=\{\lambda_g\otimes\hat{\phi}_g\}_{g\in G}$ and $\lambda\otimes \pi=\{\lambda_g\otimes \pi_g\}_{g\in G}$ factor through the reduced cross-sectional $\Cst$\nb-algebras $\Cst_r(\Hilm[B])$ and $\Cst_r(\Hilm[C])$, respectively. So let $\varphi_{\Hilm[B]}\colon\Cst_r(\Hilm[B])\to\Bound(\ell^2(G)\otimes\Hilm[K])$ and $ \varphi_{\Hilm[C]}\colon\Cst_r(\Hilm[C])\to\Bound(\ell^2(G)\otimes\Hilm[H])$ be the induced \Star homomorphisms. Then $\varphi_{\Hilm[C]}$ is faithful because $\pi=\{\pi_g\}_{g\in G}$ is so. We thus obtain a cp map $\phi'\colon\Cst_r(\Hilm[B])\to\Bound(\ell^2(G)\otimes\Hilm[H])$ by composing the cp map 
\begin{equation*}
\begin{aligned}
\Bound(\ell^2(G)\otimes \Hilm[K])&\to\Bound(\ell^2(G)\otimes\Hilm[H])\\
T&\mapsto (1\otimes V^*) T (1\otimes V)
\end{aligned}
\end{equation*} with $\varphi_{\Hilm[B]}$. Notice that $$\|\phi'\|=\|(1\otimes V)^*(1\otimes V)\|=\|V^*V\|=\|\phi\|.$$ Because $(\lambda_g\otimes\id_{\Hilm[K]})(1\otimes V)=(1\otimes V)(\lambda_g\otimes\id_{\Hilm[H]})$ for all $g\in G$, the range of $\phi'$ is contained in~$\varphi_{\Hilm[C]}(\Cst_r(\Hilm[C]))$. So we may view $\phi'$ as a cp map from $\Cst_r(\Hilm[B])$ into $\Cst_r(\Hilm[C])$ using that $\varphi_{\Hilm[C]}\colon\Cst_r(\Hilm[C])\to\Bound(\ell^2(G)\otimes \Hilm[H])$ is faithful. With this identification, we have $\pi\circ \phi'|_{\Hilm[B]} = \phi$ as wished. 
\end{proof}
\end{prop}

\begin{rem} An analogue of Proposition \ref{prop:ccp-reduced} above also holds with the full cross-sectional $\Cst$\nb-algebras in place of the reduced ones. Our proof requires a theory of cp maps between Fell bundles and a version of Stinespring's Dilation Theorem in this context. We include these results in Appendix~\ref{sec:full-Stinespring}.
\end{rem}

We briefly recall a useful description of the multiplicative domain of a ccp map. Let $\varphi\colon A\to \Bound(\Hilm[H])$ be a ccp map. The \emph{multiplicative domain} of~$\varphi$ is the set $$\{a\in A\mid \varphi(ab)=\varphi(a)\varphi(b)\text{ and }\varphi(ba)=\varphi(b)\varphi(a), \forall b\in A\}.$$ This is clearly a closed subspace of~$A$. In fact, it is also a $\Cst$\nb-subalgebra of~$A$ and it coincides with $$A_{\varphi}\coloneqq\{a\in A\mid \varphi(a)^*\varphi(a)=\varphi(a^*a)\text{ and }\varphi(a)\varphi(a)^*=\varphi(aa^*)\}.$$ See \cite[Proposition~1.5.7]{Brown-Ozawa:Approximations}. A crucial step in the proof of our main theorem, Theorem~\ref{thm:maintheorem} below, is the construction of a certain equivariant conditional expectation $\ell^\infty(G)\otimes A''\to A''$. As for global actions, the exactness hypothesis is not needed in this part and so we prove the existence of such a map in the next lemma.

\begin{lem} Let $\alpha=(\{A_g\}_{g\in G},\{\alpha_g\}_{g\in G})$ be a partial action of a discrete group~$G$ on a commutative $\Cst$\nb-algebra~$A$. Suppose that~$\alpha$ satisfies the weak containment property. Then there is a ccp map $$\phi\colon  (\ell^{\infty}(G)\otimes A'')\rtimes_{\tau\otimes\alpha'',\red} G\to A''\rtimes_{\alpha'',\red}G$$ such that $\phi((\ell^\infty(G)\otimes A_g'')\delta_g)\subset A''_g\delta_g$ and $\phi((1\otimes a)\delta_g)=a\delta_g$ for all $g\in G$ and $a\in A_g''$. \label{lem:bidual-equiv-cp} 
\begin{proof} The inclusion $A\to \ell^\infty(G)\otimes A,\ a\mapsto 1\otimes a,$ induces an inclusion of reduced crossed products $A\rtimes_{\alpha,\red}G\into (\ell^\infty(G)\otimes A)\rtimes_{\tau\otimes \alpha,\red}G, $ that is in fact an inclusion of the full crossed product $A\rtimes_\alpha G$ by the hypothesis.
We also have the canonical inclusion $A\rtimes_{\alpha,\red} G\into A''\rtimes_{\alpha'',\red}G $ and a canonical \Star homomorphism $A\rtimes_\alpha G\to A''\rtimes_{\alpha''}G$ mapping $a\delta_g\in A_g\delta_g$ to its image $a\delta_g$ in $A_g''\delta_g$ which is in fact an inclusion because every covariant representation $(\pi,u)$ for $\alpha$ gives the covariant pair $(\pi'',u)$ of $\alpha''.$ Similarly, $(\ell^\infty(G)\otimes A)\rtimes_{\tau\otimes\alpha}G$ embeds in $(\ell^\infty(G)\otimes A'')\rtimes_{\tau\otimes\alpha''}G$ via a \Star homomorphism that sends $c\delta_g$ to $c\delta_g$ since every nondegenerate representation of~$\ell^\infty(G)\otimes A$ has the form $\pi_1\times\pi_2$, where $\pi_1$ and $\pi_2$ are nondegenerate representations of~$\ell^\infty(G)$ and $A$, respectively, with commuting ranges (see \cite[Theorem~3.2.6]{Brown-Ozawa:Approximations}).
 
Consider a faithful nondegenerate representation~$\tilde{\pi}$ of~$A\rtimes_{\alpha}G$ on a Hilbert space~$\Hilm[K]$. There is a unique covariant representation $(\pi, w)$ of $\alpha=(\{A_g\}_{g\in G},\{\alpha_g\}_{g\in G})$ so that $\tilde{\pi}=\pi\times w$ \cite[Theorem~13.2]{Exel:Partial_dynamical}. Let $\lambda\colon G\to\Bound(\ell^2(G))$ be the left regular representation. By Fell's absorption principle for Fell bundles \cite[Proposition~18.4]{Exel:Partial_dynamical}, the integrated form $(1\otimes\pi)\times( \lambda\otimes w)$ induces a faithful representation of $A\rtimes_{\alpha,\red} G$ in~$\Bound(\ell^2(G)\otimes\Hilm[K])$. Using the assumption that the full and reduced partial crossed products built from~$\alpha$ coincide, we deduce that $(1\otimes\pi)\times (\lambda\otimes w)$ is a faithful representation of~$A\rtimes_{\alpha} G$.

Now let $(A'', \Hilm[H], J, P)$ be the Haagerup standard form of the enveloping von Neumann algebra~$A''$ of~$A$ and let $\iota\colon A''\to\Bound(\Hilm[H])$ be the inclusion. By Proposition~\ref{prop:cov-standard}, there is a \Star partial representation $v\colon G\to \Bound(\Hilm[H])$ so that $(\iota, v)$ is a covariant pair for the bidual partial action $\alpha''=(\{A_g''\}_{g\in G}, \{\alpha_g''\}_{g\in G})$. Let $\iota\times v\colon A''\rtimes_{\alpha''} G\to\Bound(\Hilm[H])$ be the integrated form of~$(\iota, v)$. By Arveson's extension theorem, there is a ccp map $\varphi'\colon \Bound(\ell^2(G)\otimes\Hilm[K])\to\Bound(\Hilm[H])$ such that $$\varphi'\circ ((1\otimes\pi)\times(\lambda\otimes w))=\iota\times v|_{A\rtimes_\alpha G}.$$

Consider the representation of the partial crossed product $(\ell^{\infty}(G)\otimes A)\rtimes_{\tau\otimes\alpha} G$ in~$\Bound(\ell^2(G)\otimes \Hilm[K])$ determined by the covariant pair $(M\otimes\pi, \lambda\otimes w)$, where $M$ denotes the multiplication representation of $\ell^\infty(G)$ into $\Bound(\ell^2(G))$. So the composite $$(\ell^{\infty}(G)\otimes A)\rtimes_{\tau\otimes\alpha} G\xrightarrow{(M\otimes\pi)\times(\lambda\otimes w)}\Bound(\ell^2(G)\otimes\Hilm[K])\xrightarrow{\varphi'}\Bound(\Hilm[H])$$ provides a ccp map from $(\ell^{\infty}(G)\otimes A)\rtimes_{\tau\otimes\alpha} G$
into~$\Bound(\Hilm[H])$,  denoted by~$\varphi$,  that equals $\iota\times v$ on $A\rtimes_\alpha G$. In particular, by the observation before the statement of the lemma, $A\rtimes_\alpha G$ lies in the multiplicative domain of~$\varphi$.  We will use this to show that the image of $(\ell^{\infty}(G)\otimes A_g)\delta_g$ under~$\varphi$ is contained in~$\iota(A_g'')v_g$.

Let $c\in \ell^\infty(G)\otimes A$ and $m\in A''$. Let $(a_j)_{j\in J}\subset\iota(A)$ be a net converging strongly to~$m$ and let $\xi, \eta\in\Hilm[H]$. Since $(\iota(a_j))_{j\in J}$ converges weakly to $\iota(m)$, we have
 \begin{equation*}
 \begin{aligned}
 \braket{\varphi(c\delta_e)\iota(m)\xi}{\eta}
    &=\lim_{\substack{j}}\braket{\varphi(c\delta_e)\iota(a_j)\xi}{\eta}
      =\lim_{\substack{j}}\braket{\varphi(ca_{j}\delta_e)\xi}{\eta}\\
    &=\lim_{\substack{j}}\braket{\varphi(a_{j} c\delta_e)\xi}{\eta}
     =\lim_{\substack{j}}\braket{\iota(a_{j} )\varphi(c\delta_e)\xi}{\eta}\\
    &=\braket{\iota(m)\varphi(c\delta_e)\xi}{\eta}.
\end{aligned}
\end{equation*} This implies that $\varphi(c\delta_e)$ belongs to the commutant $\iota(A'')'$.
But $A''$ is commutative because $A$ is, thus it follows from conditions (1) and (2) of Theorem~\ref{thm:standardform}
that $\iota(A'')$ is a maximal abelian von Neumann algebra in~$\Bound(\Hilm[H])$. This entails $\varphi(c\delta_e)\in \iota(A'')$ and so $\varphi(\ell^{\infty}(G)\otimes A)\subseteq A''$.

Take an approximate identity $(u_{j})_{j\in J}$ for~$A_g$ and let $c\in \ell^\infty(G)\otimes A_g$. Then
\begin{equation*}
\begin{aligned}
\varphi(c\delta_g)
  & = \lim_{\substack{j}}\varphi(c\delta_e (1\otimes u_{j})\delta_g)
    = \lim_{\substack{j}}\varphi(c\delta_e)\iota(u_j)v_g\\
  & = \lim_{\substack{j}}\varphi(c(1\otimes u_{j})\delta_e)v_g
    = \varphi(c\delta_e)v_g. 
\end{aligned}
\end{equation*}
Since we have $\varphi(c\delta_e)v_g\in \iota(A_g'')v_g,$ we conclude that $\varphi((\ell^\infty(G)\otimes A_g)\delta_g)\subseteq \iota(A_g'')v_g$. 

We claim that $\varphi$ extends to a ccp map $\varphi''\colon (\ell^\infty(G)\otimes A'')\rtimes_{\tau\otimes\alpha''}G\to\Bound(\Hilm[H])$ so that $$\varphi''((1\otimes a)\delta_g)=(\iota\times v)(a\delta_g)=\iota(a)v_g$$  for all $g\in G$ and $a\in A_g''$. To prove this,  we take a Stinespring dilation $(\tilde{\pi},\tilde{\Hilm[H]}, V)$ of~$\varphi$. We may assume that $(\tilde{\pi},\tilde{\Hilm[H]}, V)$ is a minimal dilation in the sense that $\tilde{\pi}((\ell^\infty(G)\otimes A)\rtimes_{\tau\otimes\alpha}G)V\Hilm[H]$ is dense in~$\tilde{\Hilm[H]}.$ Thus $\tilde{\pi}$ is nondegenerate and there is a unique covariant representation 
$(\pi, w)$ of~$\tau\otimes\alpha$ in~$\Bound(\tilde{\Hilm[H]})$ so that $\tilde{\pi}=\pi\times w$. Because~$\pi$ is nondegenerate, there are nondegenerate \Star homomorphisms $$\pi_1\colon\ell^\infty(G)\to\Bound(\tilde{\Hilm[H]})\qquad \text{ and }\qquad\pi_2\colon A\to\Bound(\tilde{\Hilm[H]}),$$ with commuting ranges, such that $\pi(f\otimes a)=\pi_1(f)\pi_2(a)$ for all~$f\in\ell^\infty(G)$ and~$a\in  A$. 

Now let $\pi_2''\colon A''\to\Bound(\tilde{\Hilm[H]})$ be the unique normal extension of~$\pi_2$.  Then $(\pi_1\times\pi_2'',w)$ is a covariant representation of~$\tau\otimes\alpha''$. Let $\varphi''\colon (\ell^\infty(G)\otimes A'')\rtimes_{\tau\otimes\alpha''}G\to\Bound(\Hilm[H])$ be given by the composite of the ccp map
\begin{equation*}
\begin{aligned}
\Bound( \tilde{\Hilm[H]})&\to\Bound(\Hilm[H])\\
T&\mapsto  V^*T  V
\end{aligned}
\end{equation*} with $(\pi_1\times\pi_2'')\times w$. So~$\varphi''$ extends $\varphi$ and $\varphi''=\iota\times v$ on $(1\otimes A''_g)\delta_g$ for all $g\in G$ because~$\pi_2''$ and~$\iota$ are normal representations of~$A''$. This proves our claim.

It follows that  $\varphi''((\ell^\infty(G)\otimes A''_g)\delta_g)\subset \iota(A''_g)v_g$ for all~$g\in G$. Combining this with Proposition~\ref{prop:ccp-reduced}, we deduce that there is a ccp map $$\phi\colon  (\ell^{\infty}(G)\otimes A'')\rtimes_{\tau\otimes\alpha'',\red} G\to A''\rtimes_{\alpha'',\red}G$$ such that $\phi((1\otimes a)\delta_g)=a\delta_g$ for all $g\in G$ and $a\in A_g''$. This completes the proof of the lemma.
\end{proof}
\end{lem}

\begin{cor} If $G$ is an exact discrete group and $\alpha$ is a partial action of $G$ on a commutative \cstar{}algebra~$A$ such that $A\rtimes_\alpha G=A\rtimes_{\alpha,\red}G$, then $A''\rtimes_{\alpha'',\red}G$ is nuclear.
\label{cor:bidual-commut}
\begin{proof} Let $\phi\colon  (\ell^{\infty}(G)\otimes A'')\rtimes_{\tau\otimes\alpha'',\red} G\to A''\rtimes_{\alpha'',\red}G$ be the ccp map built in Lemma~\ref{lem:bidual-equiv-cp}. We may regard~$\phi$ as a conditional expectation onto the copy of $A''\rtimes_{\alpha'',\red}G$ in $( \ell^\infty(G)\otimes A'')\rtimes_{\tau\otimes\alpha'',\red}G$. This latter $\Cst$\nb-algebra is nuclear by Corollary~\ref{cor:exact-nuclearity-diag}, since $G$ is exact and $A''$ is commutative. We conclude that $A''\rtimes_{\alpha'',\red}G$ is nuclear.
\end{proof}
\end{cor}

The next is our main theorem. It generalises a result for global actions due to Matsumura~\cite{Matsumura} -- recently extended to general locally compact groups by  Buss, Echterhoff and Willett \cite{buss2020amenability} --
to the context of partial actions of exact discrete groups on commutative $\Cst$\nb-algebras.

\begin{thm}
Let $\alpha=(\{A_g\}_{g\in G},\{\alpha_g\}_{g\in G})$ 
be a partial action of an exact discrete group~$G$ on a commutative $\Cst$\nb-algebra~$A$. Suppose that 
$A\rtimes_{\alpha}G=A\rtimes_{\alpha, \red}G$. Then the partial crossed product $A\rtimes_{\alpha}G$ is nuclear, or equivalently, $\alpha$ has the approximation property.
\label{thm:maintheorem}
\end{thm}
\begin{proof} By Corollary~\ref{cor:bidual-commut}, $A''\rtimes_{\alpha'',\red}G$ is nuclear. In particular, $A''\rtimes_{\alpha'',\red}G=A''\rtimes_{\alpha''}G$. To conclude that $A\rtimes_{\alpha}G$ is nuclear as well, let $\tilde{\pi}\colon A\rtimes_{\alpha}G\to\Bound(\Hilm[H])$ be a nondegenerate representation and let $(\pi, v)$ be a covariant pair for $\alpha$ with $\tilde{\pi}=\pi\times v$. Then $(\pi, v)$ gives the covariant pair $(\pi'', v)$ for $\alpha''$, where $\pi''$ is the unique normal extension of $\pi$. Since $\pi''\times v$ extends $\pi\times v$ and $$(\pi''\times v)(A''\rtimes_{\alpha'',\red} G)\subset ((\pi\times v)(A\rtimes_{\alpha} G))'',$$  \cite[Proposition~3.6.6]{Brown-Ozawa:Approximations} implies that for every $\Cst$\nb-algebra $B$ we have a canonical inclusion \[(A\rtimes_{\alpha} G)\otimes_\max B\hookrightarrow (A''\rtimes_{\alpha'',\red} G)\otimes_\max B.\] This shows that $A\rtimes_{\alpha} G$ is nuclear as wished. One could also deduce that $A\rtimes_{\alpha}G$ is nuclear from \cite[Theorem~6.12]{2019arXiv190703803A} as $\alpha''$ has the approximation property because $A''\rtimes_{\alpha'',\red}G$ is nuclear, and so the implication (ii)$\Rightarrow$(v) of  \cite[Theorem~6.12]{2019arXiv190703803A} gives that $\alpha$ has the approximation property as well.
\end{proof}

\subsection{A general characterisation} Our next goal is to give a characterisation of the approximation property for partial actions of exact discrete groups on $\Cst$\nb-algebras in terms of isomorphisms between full and reduced partial crossed products associated to diagonal partial actions, in a similar fashion as proved by Matsumura in \cite{Matsumura} and recently improved in \cite[Theorem~5.16]{buss2020amenability} by Buss, Echterhoff and Willett.

\begin{lem}\label{lem:bidual-ext} Let $(\{A_g\}_{g\in G},\{\alpha\}_{g\in G})$ be a partial action of a discrete group on a $\Cst$\nb-algebra~$A$. Suppose that  $$(A\otimes_{\mathrm{max}} A^{\mathrm{op}})\rtimes_{\alpha\otimes\alpha^{\mathrm{op}}}G=(A\otimes_{\mathrm{max}} A^{\mathrm{op}})\rtimes_{\alpha\otimes\alpha^{\mathrm{op}}, \red}G.$$ Then there is a ccp map $\varphi''\colon(\ell^\infty(G)\otimes A'')\rtimes_{\tau\otimes\alpha'',\red}G\to A''\rtimes_{\alpha'', \red}G$ such that for all~$g\in G$ and $a\in A_g'',$ one has
$$\varphi''((\ell^\infty(G)\otimes A''_g)\delta_g)\subset A_g''\delta_g\qquad\text{and}\qquad\varphi''((1\otimes a)\delta_g)=a\delta_g.$$  
\begin{proof} Let $(A'', \Hilm[H], J,P)$ be the Haagerup standard form of~$A''$. Let $(\iota,v)$ and $(\iota^{\mathrm{op}},v)$ be the covariant $\Wst$\nb-representations of~$(\{A''_g\}_{g\in G},\{\alpha''_g\}_{g\in G})$ and $(\{(A''_g)^{\mathrm{op}}\}_{g\in G},\{(\alpha''_g)^{\mathrm{op}}\}_{g\in G})$ in~$\Bound(\Hilm[H])$ constructed in Proposition~\textup{\ref{prop:cov-standard}}. Then the pair $(\iota\times\iota^{\mathrm{op}},v)$ gives a covariant representation of~$(\{A''_g\otimes_{\mathrm{max}}(A''_g)^{\mathrm{op}}\}_{g\in G},\{\alpha''_g\otimes (\alpha_g'')^\mathrm{op}\}_{g\in G})$ on~$\Hilm[H]$. Using that the full and reduced partial crossed products associated to~$(\{A_g\otimes_{\mathrm{max}}A_g^{\mathrm{op}}\}_{g\in G},\{\alpha_g\otimes \alpha_g^\mathrm{op}\}_{g\in G})$ coincide, we can construct a ccp map $$\phi\colon(\ell^\infty(G)\otimes(A\otimes_{\mathrm{max}}A^{\mathrm{op}}))\rtimes_{\tau\otimes(\alpha\otimes\alpha^{\mathrm{op}})}G\to\Bound(\Hilm[H])$$ as in the proof of Lemma~\ref{lem:bidual-equiv-cp} such that $\phi=(\iota\times\iota^{\mathrm{op}})\times v$ on~$(A\otimes_{\mathrm{max}} A^{\mathrm{op}})\rtimes_{\alpha\otimes\alpha^{\mathrm{op}}}G.$ Let~$\tilde{\Hilm[H]}$ be a Hilbert space, $V\colon \Hilm[H]\to\tilde{\Hilm[H]}$ a contractive operator, and $$\tilde{\pi}\colon  (\ell^\infty(G)\otimes(A\otimes_{\mathrm{max}}A^{\mathrm{op}}))\rtimes_{\tau\otimes(\alpha\otimes\alpha^{\mathrm{op}})}G\to\Bound(\tilde{\Hilm[H]})$$ a representation such that $(\tilde{\pi},\tilde{\Hilm[H]}, V)$ is a minimal Stinespring dilation of~$\phi$. This means that $\phi(b)=V^*\tilde{\pi}(b)V$  for all~$b\in (\ell^\infty(G)\otimes(A\otimes_{\mathrm{max}}A^{\mathrm{op}}))\rtimes_{\tau\otimes(\alpha\otimes\alpha^{\mathrm{op}})}G$ and $\tilde{\pi}(\exopp)V\Hilm[H]$ is dense in~$\tilde{\Hilm[H]}$. In particular, there is a unique covariant representation 
$(\pi, w)$ of~$\actex$ on~$\tilde{\Hilm[H]}$ so that $\tilde{\pi}=\pi\times w$. By nondegeneracy of~$\pi$, there are nondegenerate \Star homomorphisms $$\pi_1\colon\ell^\infty(G)\to\Bound(\tilde{\Hilm[H]})\qquad \text{ and }\qquad\pi_2\colon A\otimes_{\mathrm{max}}A^\mathrm{op}\to\Bound(\tilde{\Hilm[H]}),$$ with commuting ranges, such that $\pi(f\otimes c)=\pi_1(f)\pi_2(c)$ for all~$f\in\ell^\infty(G)$ and~$c\in  A\otimes_{\mathrm{max}}A^\mathrm{op}$.

Now let $\pi_2''\colon  (A\otimes_{\mathrm{max}}A^\mathrm{op})''\to\Bound(\tilde{\Hilm[H]})$ be the unique normal extension of~$\pi_2$. Then $(\pi_1\times\pi_2'',w)$ is a covariant representation of~$(\{\ell^\infty(G)\otimes (A_g\otimes_{\mathrm{max}}A_g^\mathrm{op})''\}_{g\in G},\{\tau_g\otimes(\alpha_g\otimes\alpha_g^\mathrm{op})''\}_{g\in G})$ in~$\Bound(\tilde{\Hilm[H]})$. We claim that for all~$g\in G$ and $a\in A_g''$, $$\pi_2''(a\otimes 1_g)=\pi_2''(a\otimes 1).$$ Here we are using the canonical normal embedding~$A''\hookrightarrow (A\otimes_{\mathrm{max}}A^\mathrm{op})''$ extending the inclusion $a\mapsto a\otimes 1$ of $A$ in~$(A\otimes_{\mathrm{max}}A^\mathrm{op})''$. We will prove that $\pi''_2(a\otimes 1)=\pi''_2(a\otimes 1_g)$ for all $a\in A_g$, and the equality for an arbitrary~$a$ in $A_g''$ will follow because $A_g$ is strongly dense in~$ A_g''$.

Let $a\in A_g$. Take $f_1,f_2\in\ell^\infty(G)$, $h,k\in G$, $b_1\in A_h,c_1\in A_h^\mathrm{op}$, $b_2\in A_k$ and $c_2\in A_k^\mathrm{op}$. Let $\xi, \eta\in\Hilm[H]$ and $(v_{j})_{j\in J}$ be an approximate identity for~$A_g$. Since $\pi''_2$ and $\iota$ are normal representations of~$(A\otimes_{\mathrm{max}}A^\mathrm{op})''$ and $A''$, respectively, we have
\begin{equation}\label{eq:factorization}\tag{$\dagger$}
\begin{aligned}
&\braket{\pi''_2(a\otimes 1_g)\pi(f_1\otimes(b_1\otimes c_1))w_hV\xi}{\pi(f_2\otimes (b_2\otimes c_2))w_kV\eta}\\&=\lim_{\substack{j}}\braket{\pi_2''(a\otimes v_j)\pi(f_1\otimes(b_1\otimes c_1))w_hV\xi}{\pi(f_2\otimes (b_2\otimes c_2))w_kV\eta}\\&=\lim_{\substack{j}}\braket{\pi(1\otimes(a\otimes v_j))(\pi(f_1\otimes(b_1\otimes c_1))w_hV\xi}{\pi(f_2\otimes (b_2\otimes c_2))w_kV\eta}\\&=\lim_{\substack{j}}\braket{V^*\pi(\tau_{k^{-1}}(f_2^*f_1)\otimes \alpha_{k^{-1}}(b_2^*ab_1)\otimes\alpha_{k^{-1}}^\mathrm{op}(c_1v_{j}c_2^*))w_{k^{-1}h}V\xi}{\eta}\\&=\lim_{\substack{j}}\braket{\phi(\tau_{k^{-1}}(f_2^*f_1)\otimes \alpha_{k^{-1}}(b_2^*ab_1)\otimes\alpha_{k^{-1}}^\mathrm{op}(c_1v_{j}c_2^*)\delta_{k^{-1}h})\xi}{\eta}\\&=\lim_{\substack{j}}\braket{\phi(1\otimes\alpha_{k^{-1}} (b_2^*)\otimes \alpha_{k^{-1}}^\mathrm{op}(c_2^*)\delta_{k^{-1}})\phi(1\otimes a\otimes v_{j})\phi((f_1^*f_2\otimes b_1\otimes c_1)\delta_h)\xi}{\eta}\\&=\lim_{\substack{j}}\braket{\phi(1\otimes\alpha_{k^{-1}}( b_2^*)\otimes \alpha_{k^{-1}}^\mathrm{op}(c_2^*)\delta_{k^{-1}})\iota(a)J\iota(v_{j})J\phi((f_1^*f_2\otimes b_1\otimes c_1)\delta_h)\xi}{\eta}\\&=\braket{\phi(1\otimes\alpha_{k^{-1}} (b_2^*)\otimes \alpha_{k^{-1}}^\mathrm{op}(c_2^*)\delta_{k^{-1}})\iota(a)\iota(1_g)\phi((f_1^*f_2\otimes b_1\otimes c_1)\delta_h)\xi}{\eta}\\&=\braket{\phi(1\otimes\alpha_{k^{-1}}( b_2^*)\otimes \alpha_{k^{-1}}^\mathrm{op}(c_2^*)\delta_{k^{-1}})\iota(a)\phi((f_1^*f_2\otimes b_1\otimes c_1)\delta_h)\xi}{\eta}.
\end{aligned}
\end{equation}
We used above that $\opp$ is contained in the multiplicative domain of~$\phi$. Now let $(u_{j})_{j\in J}$ be an approximate identity for~$A$. Reversing the computation in~\eqref{eq:factorization}, we obtain
\begin{equation*}
\begin{aligned}
\eqref{eq:factorization}&=\lim_{\substack{j}}\braket{\phi(1\otimes\alpha_{k^{-1}}( b_2^*)\otimes \alpha_{k^{-1}}^\mathrm{op}(c_2^*)\delta_{k^{-1}})\iota(a)J\iota(u_{j})J\phi((f_1^*f_2\otimes b_1\otimes c_1)\delta_h)\xi}{\eta}\\
&=\lim_{\substack{j}}\braket{\phi(1\otimes\alpha_{k^{-1}}( b_2^*)\otimes \alpha_{k^{-1}}^\mathrm{op}(c_2^*)\delta_{k^{-1}})\phi(1\otimes a\otimes u_{j})\phi((f_1^*f_2\otimes b_1\otimes c_1)\delta_h)\xi}{\eta}
\\&=\lim_{\substack{j}}\braket{\phi(1\otimes\alpha_{k^{-1}}( b_2^*)\otimes \alpha_{k^{-1}}^\mathrm{op}(c_2^*)\delta_{k^{-1}})\phi((f_1^*f_2\otimes ab_1\otimes u_{j} c_1)\delta_h)\xi}{\eta}
\\&=\lim_{\substack{j}}\braket{\phi(\tau_{k^{-1}}(f_2^*f_1)\otimes \alpha_{k^{-1}}(b_2^*ab_1)\otimes\alpha_{k^{-1}}^\mathrm{op}(c_1u_{j}c_2^*)\delta_{k^{-1}h})\xi}{\eta}
\\&=\lim_{\substack{j}}\braket{V^*\pi(\tau_{k^{-1}}(f_2^*f_1)\otimes \alpha_{k^{-1}}(b_2^*ab_1)\otimes\alpha_{k^{-1}}^\mathrm{op}(c_1u_{j}c_2^*))w_{k^{-1}h}V\xi}{\eta}
\\&=\lim_{\substack{j}}\braket{\pi(1\otimes a\otimes u_{j})\pi(f_1\otimes(b_1\otimes c_1))w_hV\xi}{\pi(f_2\otimes (b_2\otimes c_2))w_kV\eta}.
\\&=\braket{\pi''_2(a\otimes 1)\pi(f_1\otimes(b_1\otimes c_1))w_hV\xi}{\pi(f_2\otimes (b_2\otimes c_2))w_kV\eta}.
\end{aligned}
\end{equation*} This entails $\pi''_2(a\otimes 1_g)=\pi''_2(a\otimes 1)$ because we have taken a minimal Stinespring dilation of~$\phi$. We thus deduce that $\pi''_2(a\otimes 1_g)=\pi_2''(a\otimes 1)$ for all $g\in G$ and $a\in A_g''$. As a consequence, the restriction of~$\pi_1\times\pi_2''$ to $\ell^\infty(G)\otimes A''\otimes 1$ together with $w\colon G\to\Bound(\Hilm[K])$ give a covariant representation of~$(\{\ell^\infty(G)\otimes A_g''\}_{g\in G},\{\tau_g\otimes\alpha_g''\}_{g\in G})$. This gives rise to a ccp map $$\phi''\colon(\ell^\infty(G)\otimes A'')\rtimes_{\tau\otimes\alpha''} G\to\Bound(\Hilm[H])$$ defined by $\phi''(b)\coloneqq V^*((\pi_1\times\pi_2'')\times w)(b)V$. Also, $\phi''((1\otimes a)\delta_g)=\iota(a)v_g$ for all $g\in G$ and $a\in A_g$ since for all~$\xi,\eta\in \Hilm[H]$ we have
\begin{equation*}
\begin{aligned}\braket{\phi''((1\otimes a)\delta_g)\xi}{\eta}&=\braket{V^*\pi''_2(a\otimes 1)w_gV\xi}{\eta}\\&=\braket{V^*\pi''_2(a\otimes 1_g)w_gV\xi}{\eta}\\&=\lim_{\substack{j}}\braket{V^*\pi(1\otimes a\otimes v_{j})w_gV\xi}{\eta}\\&=\lim_{\substack{j}}\braket{\iota(a)Jv_{j}Jv_g\xi}{\eta}\\&=\braket{\iota(a)v_g\xi}{\eta},
\end{aligned}
\end{equation*}
where $(v_{j})_{j\in J}$ is an approximate identity for $A_g$. The same equality will hold for all $a\in A_g''$ because $\iota$ and $\pi_2''$ are normal representations of~$A''$ and $(A\otimes_{\mathrm{max}}A^\mathrm{op})''$, respectively. 

We claim that $\phi''((\ell^{\infty}(G)\otimes A_g'')\delta_g)\subseteq \iota(A_g'')v_g$. To show this, it suffices to prove that $\phi''(\ell^\infty(G)\otimes A)\subseteq \iota(A'')$. Let $f\in\ell^\infty(G)$ and $a\in A$. Let $(u_{j})_{j\in J}$ be an approximate identity for~$A$. Since $\pi_2(A\otimes_{\mathrm{max}}A^\mathrm{op})$ is contained in the multiplicative domain of the ccp map $\Bound(\tilde{\Hilm[H]})\to\Bound(\Hilm[H])$, $T\mapsto V^*TV$, we have for all $b\in A^{\mathrm{op}}$, and for all~$\xi,\eta\in\Hilm[H]$, 
\begin{equation*}
\begin{aligned}\braket{\phi''(f\otimes a)J\iota(b)^*J\xi}{\eta}&=\lim_{\substack{j}}\braket{\phi''(f\otimes a)\iota(u_{j})J\iota(b)^*J\xi}{\eta}\\&=\lim_{\substack{j}}\braket{V^*\pi_1(f)\pi_2''(a\otimes 1)VV^*\pi_2(u_{j}\otimes b)V)\xi}{\eta}
\\&=\braket{V^*\pi_1(f)\pi_2(a\otimes b)V)\xi}{\eta}\\&=\lim_{\substack{j}}\braket{V^*\pi_1(f)\pi_2(u_{j}a\otimes b)V\xi}{\eta}\\&=\lim_{\substack{j}}\braket{V^*\pi_2(u_{j}\otimes b)VV^*\pi_1(f)\pi_2''(a\otimes 1)V\xi}{\eta}
\\&=\lim_{\substack{j}}\braket{\iota(u_{j})J\iota(b)^*JV^*\pi_1(f)\pi_2''(a\otimes 1)V\xi}{\eta}
\\&=\braket{J\iota(b)^*J\phi''(f\otimes a)\xi}{\eta}.
\end{aligned}
\end{equation*} Hence $\phi''(\ell^{\infty}(G)\otimes A)$ commutes with $J AJ$. So it commutes with $JA''J$ because $JAJ$ is weakly dense in~$JA''J$. We deduce that $\phi''(\ell^\infty(G)\otimes A)\subseteq (JA''J)'=A''$ and thus $\phi''(\ell^\infty(G)\otimes A'')\subseteq \iota(A'')$.

In order to complete the proof of the lemma, we apply Proposition~\ref{prop:ccp-reduced}. Because $\iota\colon A''\to\Bound(\Hilm[H])$ is faithful, there is a ccp map $$\varphi''\colon (\ell^\infty(G)\otimes A'')\rtimes_{\tau\otimes\alpha'', \red}G\to A''\rtimes_{\alpha'',\red}G$$ such that $\varphi''|_{\B_{\tau\otimes\alpha''}}=\phi''$ under the identification of~$A_g''\delta_g$ with $\iota(A''_g)v_g\subset\Bound(\Hilm[H])$. Thus, for all $g\in G$ and $a\in A_g''$, we have $\varphi''((\ell^\infty(G)\otimes A''_g)\delta_g)\subset A''_g\delta_g$ and $\varphi''((1\otimes a)\delta_g)=a\delta_g$. This finishes the proof of the lemma.
\end{proof}
\end{lem}

Before we proceed with the proof of our second main result, let us deduce a direct consequence of the previous lemma regarding the restriction of~$\alpha''$ to the center of~$A''$:

\begin{cor}\label{cor:Z-weak-containment}
If $G$ is exact and $\alpha$ is a partial action of $G$ on a \cstar{}algebra~$A$ such that the diagonal partial action $\alpha\otimes \alpha^\op$ on~$A\otimes_\max A^{\mathrm{op}}$ satisfies the weak containment property, then $\mathrm{Z}(A'')\rtimes_{\alpha'',\red}G$ is nuclear.
\end{cor}
\begin{proof}  Let $\varphi''$ be the ccp map from Lemma~\ref{lem:bidual-ext}. Notice that $\varphi''((\ell^\infty(G)\otimes 1)\delta_e)\subset\mathrm{Z}(A'')\delta_e$ and so $$\varphi''((\ell^\infty(G)\otimes\mathrm{Z}(A_g''))\delta_g)\subset \mathrm{Z}(A_g'')\delta_g$$ for all $g\in G$. Thus $\varphi''$ restricts to a ccp $ (\ell^\infty(G)\otimes \mathrm{Z}(A''))\rtimes_{\tau\otimes\alpha'', \red}G\to \mathrm{Z}(A'')\rtimes_{\alpha'',\red} G$ that, when composed with the canonical embedding $$\mathrm{Z}(A'')\rtimes_{\alpha'',\red}G\hookrightarrow(\ell^\infty(G)\otimes \mathrm{Z}(A''))\rtimes_{\tau\otimes\alpha'', \red}G,$$ yields a conditional expectation onto a copy of~$\mathrm{Z}(A'')\rtimes_{\alpha'',\red}G$. Since $G$ is exact, the diagonal partial $G$-action $\tau\otimes\alpha''$ on $\ell^\infty(G)\otimes \mathrm{Z}(A'')$ has the approximation property. So the associated partial crossed product is nuclear because $\mathrm{Z}(A'')$ is commutative.  Hence $\mathrm{Z}(A'')\rtimes_{\alpha'',\red}G$ is nuclear as wanted.
\end{proof}

We observe that the equivalence between the statements in the theorem below was established in the context of global actions of exact locally compact groups by Buss, Echterhoff and 
Willett in~\cite[Theorem~5.16 and Corollary~5.17]{buss2020amenability}. The corresponding result for actions of exact discrete groups was proved earlier by Matsumura with the assumption that $A$ is unital and nuclear \cite{Matsumura}.

\begin{thm}\label{thm:noncommutative} Let~$G$ be an exact discrete group and let~$(\{A_g\}_{g\in G},\{\alpha_g\}_{g\in G})$ be a partial action of~$G$ on a $\Cst$\nb-algebra~$A$. Then the following are equivalent:

\begin{enumerate}
\item[\rm{(i)}] $\alpha$ has the approximation property;

\item[\rm{(ii)}] $(A\otimes_\max B)\rtimes_{\alpha\otimes\beta}G=(A\otimes_\max B)\rtimes_{\alpha\otimes\beta, \red}G$ for every partial action $(\{B_g\}_{g\in G},\{\beta_g\}_{g\in G})$ of~$G$ on a $\Cst$\nb-algebra~$B$;
\item[\rm{(iii)}] $(A\otimes_\max A^{\mathrm{op}})\rtimes_{\alpha\otimes\alpha^{\mathrm{op}}}G=(A\otimes_\max A^{\mathrm{op}})\rtimes_{\alpha\otimes\alpha^{\mathrm{op}}, \red}G$.
\end{enumerate}
\begin{proof} Suppose that $\alpha$ has the approximation property. By Corollary~\ref{cor:nuclearity-diag}, the diagonal partial action  $(\{A_g\otimes_\max B_g\}_{g\in G},\{\alpha_g\otimes\beta_g\}_{g\in G})$ also has the approximation property and therefore $(A\otimes_\max B)\rtimes_{\alpha\otimes\beta}G=(A\otimes_\max B)\rtimes_{\alpha\otimes\beta, \red}G$. This gives the implication $\text{(i)}\Rightarrow\text{(ii)}$, while $\text{(ii)}\Rightarrow\text{(iii)}$ is obvious. 

It remains to establish $\text{(iii)}\Rightarrow\text{(i)}.$ Since~$G$ is exact, Corollary~\ref{cor:Z-weak-containment} says that the partial crossed product $\mathrm{Z}(A'')\rtimes_{\alpha'',\red}G$ is nuclear. Hence the associated semidirect product bundle has the approximation property. Then so does $\Hilm[B]_\alpha$ by Corollary 6.10, Proposition~3.5 and Theorem 6.12 of~\cite{2019arXiv190703803A}.
\end{proof}
\end{thm}

\section{Some applications}

In this section, we combine Theorem~\ref{thm:maintheorem} and a partial crossed product picture of a semigroup $\Cst$\nb-algebra due to Li \cite[Theorem~5.6.41]{CLEY} to show that the semigroup $\Cst$\nb-algebra of a monoid~$P$ is nuclear whenever~$P$ embeds into an exact discrete group and the left regular representation of~$P$ on~$\ell^2(P)$ implements an isomorphism between full and reduced $\Cst$\nb-algebras. We also apply Theorem~\ref{thm:maintheorem} to establish an equivalence between the weak containment property and nuclearity of the $\Cst$\nb-algebras in the case of $\Cst$\nb-algebras associated to separated graphs.

\subsection{Semigroup \texorpdfstring{$\mathrm{C}^*$}{C*}-algebras}  There is a natural concrete $\Cst$\nb-algebra attached to a left-cancellative semigroup~$P$ with unit element~$e$. This is the $\Cst$\nb-subalgebra of~$\Bound(\ell^2(P))$ generated by the canonical representation of~$P$ by isometries on~$\ell^2(P)$. Li introduced a construction of a universal $\Cst$\nb-algebra~$\Cst_s(P)$ associated to a semigroup in a general context in~\cite{Li:Semigroup_amenability}, following the work of Nica on $\Cst$\nb-algebras associated to positive cones of quasi-lattice orders \cite{Nica:Wiener--hopf_operators}. We recall the definition of the universal semigroup $\Cst$\nb-algebra attached to a submonoid of a group. The defining relations involve the family of constructible right ideals of~$P$, and therefore we recall how these are defined as well.

We employ the notation from \cite[Section~4.3]{Sehnem}. Let $\alpha=(p_1,p_2,\ldots,p_{2k})$ be a word in~$P$. We define  \begin{equation*}\label{eq:finite_set} F_{\alpha}=\{p_{2k}^{-1}p_{2k-1}, p_{2k}^{-1}p_{2k-1}p_{2k-2}^{-1}p_{2k-3},\ldots, p_{2k}^{-1}p_{2k-1}p_{2k-2}^{-1}\cdots p_2^{-1}p_1\}.\end{equation*}  Set $$K_{F_{\alpha}}\coloneqq P\cap p_{2k}^{-1}p_{2k-1}P\cap p_{2k}^{-1}p_{2k-1}p_{2k-2}^{-1}p_{2k-3}P\cap\ldots\cap p_{2k}^{-1}p_{2k-1}p_{2k-2}^{-1}\cdots p_2^{-1}p_1P.$$ Then $K_{F_{\alpha}}$ is a right ideal in~$P$. This is the right ideal $$ p_{2k}^{-1}p_{2k-1}p_{2k-2}^{-1}\cdots p_2^{-1}p_1P$$ in the notation of Li~\cite{Li:Semigroup_amenability}. It follows from equation (32) of~\cite{Li:Semigroup_amenability} that $$\Hilm[J]\coloneqq\{K_{F_{\alpha}}\mid \alpha=(p_1,p_2, \ldots,p_{2k}),\,k\geq 1, p_{i}\in P,\,\text{ for } i=1,\ldots, 2k\}\cup\{\emptyset\}$$ is closed under intersection. This is the smallest family of right ideals of~$P$ containing~$P$ and~$\emptyset$ and closed under taking images and pre-images of the left multiplication action by elements in~$P$. That is, if~$S\in\mathcal{J}$ and $p\in P$, then~$pS\in\mathcal{J}$  and~$p^{-1}S\cap P\in\mathcal{J}$. We refer to~$\Hilm[J]$ as the family of "constructible" right ideals of~$P$.

The family~$\Hilm[J]$ of right ideals of~$P$ is called \emph{independent} \cite[Definition 2.26]{Li:Semigroup_amenability} if given a right ideal of~$P$ of the form $$S=\underset{i=1}{\overset{m}{\bigcup}}S_i,$$ with $ S\in\Hilm[J]$ and $S_i\in\mathcal{J}$ for all $i\in\{1,\ldots,m\}$, then there exists $i\in\{1,\ldots,m\}$ such that~$S=S_i$. If~$\Hilm[J]$ is independent, we simply say that~$P$ satisfies independence.

\begin{defn}[\cite{Li:Semigroup_amenability}*{Definition 3.2}]\label{defn:fullcstarP} Let $P$ be a submonoid of a group~$G$ and let $\Hilm[J]$ be the family of constructible ideals of $P$. The \emph{semigroup $\Cst$\nb-algebra} of~$P$, denoted by~$\Cst_s(P)$, is the universal $\Cst$\nb-algebra generated by a family of isometries~$\{v_p\mid p\in P\}$ and projections~$\{e_S\mid S\in\Hilm[J] \}$ satisfying the following:
\begin{enumerate}\label{defn:semigroupCstaralgebra}
\item[\textup{(i)}] $v_pv_q=v_{pq}$
\item[\textup{(ii)}] $e_{\emptyset}=0$
\item[\textup{(iii)}] $v_{p_1}^*v_{p_2}\cdots  v_{p_{2k-1}}^*v_{p_{2k}}=e_{K_{F_{\alpha}}}$ whenever~$\alpha=(p_1,p_2,\ldots,p_{2k})$ is a word in~$P$ satisfying $p_1^{-1}p_2\ldots p_{2k-1}^{-1}p_{2k}=e$ in~$G$.
\end{enumerate}
\end{defn}

We observe that $\Cst_s(P)$ does not depend on the embedding $P\hookrightarrow G$  (see \cite[p. 4327]{Li:Semigroup_amenability}). 

\begin{rem} It follows from Corollary~2.10 and Lemma~3.3 of~\cite{Li:Semigroup_amenability} that $\Cst_s(P)$ is generated as a $\Cst$\nb-algebra by the semigroup of isometries~$\{v_p\mid p\in P\}$. 
\end{rem}

Next we recall how the left regular representation of~$P$ in $\Bound(\ell^2(P))$ is defined. For each $p\in P$, let $V_p\colon \ell^2(P)\to \ell^2(P)$ be given by $$V_p(\delta_q)=\delta_{pq} \qquad \qquad (q\in P),$$ where $\delta_q$ denotes the characteristic function of~$\{q\}$ viewed as a vector of the canonical orthonormal basis of~$\ell^2(P)$. Since~$P$ is left-cancellative, $V_p$ is an isometry with $$
V_p^*(\delta_q)=
\begin{cases}
\delta_{p^{-1}q}& \text{if } q\in pP,\\
0 &\text{otherwise}.
\end{cases}$$ In particular, $V_e=1$. The set of isometries $\{V_p\mid p\in P\}$ satisfies the relation $$V_pV_q=V_{pq}$$ for all $p,q\in P$. Thus $p\mapsto V_p$ is an isometric representation of~$P$. In addition, it follows from \cite[Lemma~3.1]{Li:Semigroup_amenability} that if $p_1^{-1}p_2\ldots p_{2k-1}^{-1}p_{2k}=e$, then $$V_{p_1}^*V_{p_2}\cdots  V_{p_{2k-1}}^*V_{p_{2k}}=\mathbb{1}_{K_{F_{\alpha}}},$$ where $\mathbb{1}_{K_{F_{\alpha}}}\in\ell^\infty(P)$ is the characteristic function of the right ideal $K_{F_{\alpha}}$, acting on~$\ell^2(P)$ by left multiplication.

\begin{defn} The \emph{reduced semigroup $\Cst$\nb-algebra} associated to~$P$ , denoted by $\Cst_{\lambda}(P)$, is the $\Cst$\nb-subalgebra of $\Bound(\ell^2(P))$ generated by the semigroup of isometries $\{V_p\mid p\in P\}$.
\end{defn}

The map $p\mapsto V_p$ is called the \emph{left regular representation} of~$P$. It induces a surjective \Star homomorphism $\lambda^+\colon\Cst_s(P)\to\Cst_{\lambda}(P)$ mapping $v_p$ to $V_p$, which we also refer to as the left regular representation of~$\Cst_s(P)$.

We denote by $D_r\subseteq \Cst_{\lambda}(P)$ the $\Cst$\nb-subalgebra of~$\Cst_{\lambda}(P)\subseteq\Bound(\ell^2(P))$ generated by $$\{\mathbb{1}_{K_{F_{\alpha}}}\mid \alpha=(p_1,\ldots, p_{2k})\text{ is a word in }P\}.$$ There is a faithful conditional expectation $$E_r\colon \Cst_{\lambda}(P)\to D_r$$ which is determined by $$\braket{E_r(b)\delta_q}{\delta_q}=\braket{b\delta_q}{\delta_q}$$ for all $b\in \Cst_{\lambda}(P)$ and $q\in P$. By \cite[Corollary~2.27]{Li:Semigroup_amenability}, $P$ satisfies independence if and only if the restriction of the left regular representation to the $\Cst$\nb-subalgebra $D_s\subseteq \Cst_s(P)$ generated by the projections $\{e_S\mid S\in\Hilm[J] \}$  is an isomorphism onto~$D_r$.

Whenever $P$ embeds into a group~$G$, it follows from the work of Li \cite[Theorem~5.6.41]{CLEY} that $\Cst_{\lambda}(P)$ can be realised as a reduced partial crossed product associated to a partial action of~$G$ on the diagonal subalgebra~$D_r$. When $P$ satisfies independence, the corresponding full partial crossed product is then isomorphic to $\Cst_s(P)$. We will sketch the main ideas in the proof that a semigroup $\Cst$\nb-algebra can be obtained from a partial action. 

Following \cite{CLEY}, we let $I_l(P)\subseteq\Cst_{\lambda}(P)$ be the smallest semigroup of partial isometries generated by the $V_p$'s (see also \cite[Definition~5.3]{LI2013626}). Specifically, $$I_l(P)=\{V_{p_1}^*V_{p_2}\cdots  V_{p_{2k-1}}^*V_{p_{2k}}\mid p_i\in P, i\in\{1,\dots, 2k\}, k\geq 0\}.$$ For $V=V_{p_1}^*V_{p_2}\cdots  V_{p_{2k-1}}^*V_{p_{2k}}\in I_l(P)$, $V^*V=\mathbb{1}_{K_{F_{\alpha}}}$ with $\alpha=(p_1,p_2,\ldots, p_{2k})$. In particular, the range and source projections associated to the partial isometries in~$I_l(P)$ commute among themselves.

\begin{lem}[\cite{LI2013626}*{Lemma~5.4}]  Let $I_l(P)^{\times}=I_l(P)\setminus\{0\}$ and let  $P\hookrightarrow G$ be an embedding. Then the map $\sigma\colon I_l(P)^\times\to G$ which sends a partial isometry $V=V_{p_1}^*V_{p_2}\cdots  V_{p_{2k-1}}^*V_{p_{2k}}$ to $\sigma(V)=p_1^{-1}p_{2}\cdots p_{2k-1}^{-1}p_{2k}$ is well-defined. Precisely, the element $\sigma(V)$ is the unique $g\in G$ with $V(\delta_q)=\delta_{gq}$ for all $q\in P$ such that $V(\delta_q)\neq 0$. Moreover, the function $V\mapsto \sigma(V)$ satisfies $\sigma(V^*)=\sigma(V)^{-1}$ and $\sigma(V_1)\sigma(V_2)=\sigma(V_1V_2)$ whenever $V_1V_2\neq 0$.
\end{lem}

For each~$g\in G$, let $$A_{g^{-1}}\coloneqq\overline{\mathrm{span}}\{V^*V\mid \sigma(V)=g\}\subseteq D_r.$$ So $A_e=D_r$ because $V=V_{p_1}^*V_{p_2}\cdots  V_{p_{2k-1}}^*V_{p_{2k}}$ is the projection onto the ideal $K_{F_{\alpha}}$ whenever $\sigma(V)=e$, where $\alpha=(p_1, \ldots, p_{2k})$. Also, $A_{g^{-1}}$ is an ideal of~$D_r$. This is so since $V^*VV'=(V')^*V^*VV'$ for all $V'\in I_l(P)^\times$ with $\sigma(V')=e$, and $\sigma(VV')=g$ if $VV'\neq 0$. 

\begin{prop}[\cite{CLEY}*{Section 5.5.2}]\label{prop:monoid-par-action} Let $P$ be a submonoid of a group~$G$. Then there is a unique \Star isomorphism $\gamma_g\colon A_{g^{-1}}\to A_g$ mapping $V^*V\in A_{g^{-1}}$ to $VV^*\in A_g$. Moreover, $\gamma=(\{A_g\}_{g\in G},\{\gamma_g\}_{g\in G})$ is a partial action of~$G$ on~$D_r$.
\end{prop}
\begin{proof} We explain how $\gamma_g$ is constructed. View $\Bound(\ell^2(P))$ as a closed $\Cst$\nb-subalgebra of~$\Bound(\ell^2(G))$ using the canonical embedding of~$\ell^2(P)$ as a closed subspace of~$\ell^2(G)$. Let $\lambda\colon G\to\Bound(\ell^2(G))$ be the left regular representation of~$G$. Then $\lambda_gV^*V=V$ and so $\lambda_g V^*V\lambda_{g^{-1}}=VV^*$ for all $V\in I_l(P)^\times$ with $\sigma(V)=g$. Hence the map $V^*V\mapsto VV^*$  is part of a well-defined \Star isomorphism $\gamma_g\colon A_{g^{-1}}\to A_g$. Uniqueness of~$\gamma_g$ is clear. 

In order to see that $\gamma=(\{A_g\}_{g\in G},\{\gamma_g\}_{g\in G})$ is indeed a partial action on~$D_r$, notice that for all $V,W\in I_l(P)^\times$ such that $\sigma(V)=g$ and $\sigma(W)=h^{-1}$, we have \[\gamma_{g}(V^*VW^*W)=\lambda_gV^*VW^*W\lambda_{g^{-1}}=\lambda_gV^*VW^*WV^*V\lambda_{g^{-1}}=VW^*WV^*.\] This shows that $\gamma$ satisfies axiom (ii) of Definition~\ref{defn: partial-action} since $\sigma(WV^*)=(gh)^{-1}$ if $WV^*\neq 0$. Now that $\gamma$ also satisfies axiom (iii) of Definition~\ref{defn: partial-action} follows because $\gamma_{gh}$ is simply conjugation by the unitary $\lambda_{gh}$, and $\lambda_{gh}=\lambda_g\lambda_h$.
\end{proof}

\begin{thm}[\cite{CLEY}*{Theorem~5.6.41}]\label{thm:partial-picture} Let $P$ be a submonoid of a group~$G$. Let $\gamma=(\{A_g\}_{g\in G},\{\gamma_g\}_{g\in G})$  be the partial action of~$G$ on~$D_r$ from Proposition~\textup{\ref{prop:monoid-par-action}}. Then $\Cst_{\lambda}(P)\cong D_r\rtimes_{\gamma, \red}G$ with an isomorphism that sends an isometry $V_p$ to~$\mathbb{1}_{pP}\delta_p\in D_r\rtimes_{\gamma, \red}G$. If, in addition, $P$ satisfies independence, then $v_p\mapsto \mathbb{1}_{pP}\delta_p$ induces an isomorphism $$\Cst_s(P)\cong D_r\rtimes_\gamma G.$$ 
\begin{proof} We sketch the main ideas in the proof of~\cite[Theorem~5.6.41]{CLEY}. For each $g\in G$, define a closed subspace of~$\Cst_{\lambda}(P)$ by setting $$B_g\coloneqq\overline{\mathrm{span}}\{V\in I_l(P)^\times\mid \sigma(V)=g\}.$$ It follows that $\{B_g\}_{g\in G}$ is a topological $G$\nb-grading for~$\Cst_{\lambda}(P)$ since the latter carries a canonical faithful conditional expectation onto~$D_r=B_e$. See Definitions~16.2 and 19.2 of \cite{Exel:Partial_dynamical} for further details. Li identifies in~\cite[Theorem~5.6.41]{CLEY} the Fell bundle $(B_g)_{g\in G}$ with the semirect product bundle $\Hilm[B]_\gamma$ of $\gamma$, via an isomorphism that sends a partial isometry~$V$ of the generating set~$ I_l(P)\cap B_g$ to $VV^*\delta_g\in \Hilm[B]_{\gamma_g}=A_g\delta_g$, for each~$g\in G$. In particular, $V_p$ is mapped to $\mathbb{1}_{pP}\delta_p$ for all~$p\in P$. This gives rise to an isomorphism $\Cst_{\lambda}(P)\cong D_r\rtimes_{\gamma,\red}G$ by \cite[Propostion~19.8]{Exel:Partial_dynamical}  because $E_r\colon \Cst_{\lambda}(P)\to D_r$ is a faithful conditional expectation.

We deduce by universal property of~$\Cst_s(P)$ that the map which sends $v_p$ to $\mathbb{1}_{pP}\delta_p$ and a projection $e_{K_{F_{\alpha}}}\in D_s$ to~$\mathbb{1}_{K_{F_{\alpha}}}\delta_e$ yields a surjective \Star homomorphism $\psi\colon \Cst_s(P)\to D_r\rtimes_\gamma G$.  Assume that $P$ satisfies independence. Then $\psi$ is injective on~$D_s$ by \cite[Corollary~2.27]{Li:Semigroup_amenability}. In order to construct an inverse for~$\psi$, consider, for each~$g\in G$, the closed subspace of~$\Cst_s(P)$ given by $$B'_g\coloneqq\overline{\mathrm{span}}\{v_{p_{1}}^*v_{p_2}\ldots v_{p_{2k-1}}^*v_{p_{2k}}\mid p_1^{-1}p_2\cdots p_{2k-1}^{-1}p_{2k}=g\}.$$ It follows that the left regular representation $\lambda^+\colon \Cst_s(P)\to \Cst_{\lambda}(P)$ restricts to an isomorphism from $B_g'$ onto~$B_g\subseteq\Cst_{\lambda}(P)$ for all~$g\in G$, since~$\lambda^+$ is faithful on~$B'_e=D_s$ and $b^*b\in B'_e$ for all~$b\in B'_g$. As for the reduced semigroup $\Cst$\nb-algebra, the collection of closed subspaces $\{B'_g\}_{g\in G}$ forms a topological $G$\nb-grading for~$\Cst_s(P)$. 

To establish an isomorphism between $\Cst_s(P)$ and the full partial crossed product $D_r\rtimes_{\gamma}G$ when $P$ satisfies independence, we let $\tilde{\pi}$ denote the representation of~$D_r\rtimes_{\gamma}G$ on~$\ell^2(P)$ obtained from the isomorphism of Fell bundles $(B_g)_{g\in G}$ and $\Hilm[B]_\gamma$. Let $(\pi,w)$ be the unique covariant pair so that $\tilde{\pi}=\pi\times v$. We can define a representation $\psi'$ of $\gamma=(\{A_g\}_{g\in G},\{\gamma_g\}_{g\in G})$ in $\Cst_s(P)$ by setting $$\psi'(a\delta_g)=(\lambda_g^+)^{-1}((\pi\times w)(a\delta_g))$$ for all $g\in G$ and $a\in A_g$. Here $(\lambda_g^+)^{-1}$ stands for the inverse of $\lambda^+\restriction_{B'_g}\colon B'_g\to B_g$. So~$\psi'$ gives a representation of $\Hilm[B]_\alpha$ satisfying $\psi'(\mathbb{1}_{pP}\delta_p)=v_p$ for all~$p\in P$. The induced \Star homomorphism $\tilde{\psi'}\colon D_r\rtimes_\alpha G\to\Cst_s(P)$ is the desired  inverse of~$\psi$.
\end{proof}
\end{thm}

The following answers affirmatively a question posed by Anantharaman-Delaroche in \cite[Remark~4.17]{anantharamandelaroche2016remarks}.

\begin{cor} Let $P$ be a monoid that embeds into an exact discrete group~$G$. Suppose that the left regular representation $\lambda^+\colon \Cst_s(P)\to\Cst_{\lambda}(P)$ is an isomorphism. Then $\Cst_{\lambda}(P)$ is nuclear. 
\begin{proof} Since the left regular representation is an isomorphism, $P$ satisfies independence by \cite[Corollary~2.27]{Li:Semigroup_amenability}. Nuclearity of $\Cst_{\lambda}(P)$ then follows as a combination of Theorems~\ref{thm:partial-picture} and~\ref{thm:maintheorem}.
\end{proof}
\end{cor}

\begin{rem} There are subsemigroups of exact discrete groups that satisfy independence but $\Cst_s(P)\neq\Cst_{\lambda}(P)$. The nonabelian Artin monoid of finite type $A^+_M$ associated to a Coxeter matrix~$M$ embeds into the corresponding Artin group~$A_M$ by the work of Brieskorn and Saito \cite{Brieskorn1972}. Exactness for Artin groups of finite type follows from~\cite{Cohen2002}. That $\Cst_s(A_M^+)\neq \Cst_\lambda(A_M^+)$ is \cite[Theorem~30]{crisp_laca_2002}.
\end{rem}

\subsection{Separated graph \texorpdfstring{$\mathrm{C}^*$}{C*}-algebras}

Let $(E,C)$ be a finitely separated graph with source and range maps denoted by $s,r\colon E^1\to E^0$. We do not need to define all the notations involved here, but just recall that $(E,s,r)$ is a directed graph and $C$ is a separation on the edges, that is, a collection of partitions $C_v$ of $s^{-1}(v)$ for each vertex $v\in E^0$. To a separated graph we can attach some $\Cst$\nb-algebras, most notably the tame $\Cst$\nb-algebra $\OO(E,C)$ and its reduced version $\OO_\red(E,C)$, which is a quotient of $\OO(E,C)$. We refer to \cite{Ara-Goodearl:C-algebras_separated_graphs, Ara-Exel-Katsura:Dynamical_systems} for further details. Using the description of these algebras as partial crossed products as obtained in \cite{Lolk:Exchange} we get from our main result the following:

\begin{cor}\label{cor:separated-graph}
The quotient map $\OO(E,C)\to \OO_\red(E,C)$ is an isomorphism if and only if one (and hence both) of these $\Cst$\nb-algebras is nuclear.
\end{cor}
\begin{proof}
It is enough to know that $\OO(E,C)$ is a partial crossed product of the form $C_0(X)\rtimes_\alpha G$, where $X$ is a locally compact (totally disconnected) Hausdorff space, $G$ is a free group (on the edges of the graph), and $\alpha$ is a partial action of $G$ on $C_0(X)$. The nature of $X$ is not relevant to us here. The reduced tame \cstar{}algebra $\OO_\red(E,C)$ is then the corresponding reduced crossed product $C_0(X)\rtimes_{\alpha,\red} G$. Since free groups are exact, our main result, Theorem~\ref{thm:maintheorem}, implies the desired result.
\end{proof}

\appendix

\section{Completely positive maps for Fell bundles}\label{sec:full-Stinespring}

The main purpose of this appendix is to establish a universal property for the full cross-sectional $\Cst$\nb-algebra of a Fell bundle with respect to certain completely positive maps, defined in principle only on the fibres. We then obtain an analogue of Proposition~\ref{prop:ccp-reduced} in which the resulting cp map is defined at the level of the full cross-sectional $\Cst$\nb-algebras (see Proposition~\ref{prop:universal}). We refer the reader to \cite{blecher2004operator} for the theory of completely positive maps on operator spaces and operator algebras.

\subsection{Completely positive maps} Let $\Hilm[B]=(B_g)_{g\in G}$ be a Fell bundle over a discrete group. We write $\mathbb{M}_n(\Hilm[B])$ for the set of $n\times n$ matrices with entries in~$\Hilm[B]$. The $(i,j)$ entry of a matrix $M\in\mathbb{M}_n(\Hilm[B])$ will be denoted by~$M_{i,j}$. We let $G^n$ denote the Cartesian product of~$n$ copies of~$G$ and take $\mathfrak{t}=(t_1, \ldots,t_n)\in G^n$. Following~\cite{Abadieferraro}, we set $$\mathbb{M}_{\mathfrak{t}}(\Hilm[B])\coloneqq\{M\in\mathbb{M}_n(\Hilm[B])\mid M_{i,j}\in B_{t_i^{-1}t_j}, i,j\in\{1,\ldots,n\}\}.$$ It follows from \cite[Lemma~2.8]{Abadieferraro} that $\mathbb{M}_{\mathfrak{t}}(\Hilm[B])$ is a $\Cst$\nb-algebra, with entrywise vector space operations, matrix multiplication, and involution given by $$M^*_{i,j}=(M_{j,i})^*, \quad\text{for all }i,j\in\{1,\ldots,n\}.$$ Its $\Cst$\nb-norm is inherited from $\mathbb{M}_n(\Cst(\Hilm[B]))$, or $\mathbb{M}_n(\Cst_{\red}(\Hilm[B]))$, and is equivalent to the supremum norm $M_\infty=\max_{\substack{i,j}}\|M_{i,j}\|$. 

\begin{rem}\label{rem:compactoperators} Let $B$ be a $\Cst$\nb-algebra and~$n$ a positive integer. For each $i\in\{1,\ldots, n\},$ let $b_i\in B$. Then the element $(b_i^*b_j)_{i,j}$ is positive in~$\mathbb{M}_n(B)$ (see \cite{Raeburn-Williams:Morita_equivalence}*{Lemma~2.65}). Now let $\B=(B_g)_{g\in G}$ be a Fell bundle and $\mathfrak{t}=(t_1,\ldots,t_n)\in G^n$. For each $i\in\{1,\ldots, n\}$, take $b_i\in B_{t_i}$. It follows that the matrix $(b_i^*b_j)_{i,j}$ is positive in~$\mathbb{M}_n(\Cst(\B))$, or $\mathbb{M}_n(\Cst_{\red}(\B))$. Hence it is also positive when regarded as an element of~$\mathbb{M}_{\mathfrak{t}}(\Hilm[B])$.
\end{rem}

\begin{defn} Let $\Hilm[H]$ be a Hilbert space and let $\phi\colon \Hilm[B]\to \Bound(\Hilm[H])$ be a function which is linear on each fibre of~$\Hilm[B]$. Let $\mathfrak{t}\in G^n$. We let $\phi^{\mathfrak{t}}$ be the induced linear function \begin{eqnarray*}
\phi^{\mathfrak{t}}\colon \mbb&\to& \mathbb{M}_n(\Bound(\Hilm[H]))\equiv \Bound(\Hilm[H]^n)\\
M=(M_{i,j})&\mapsto& \phi^{\mathfrak{t}}(M)=(\phi(M_{i,j})_{i,j}).
\end{eqnarray*}

We will say that $\phi$ is \emph{completely positive} if $\phi^{\mathfrak{t}}$ is positive for all $\mathfrak{t}\in\cup_{\substack{n=1}}^\infty G^n$. 

If $\phi$ is completely positive and its \emph{norm} 
	$$\|\phi\|\coloneqq\sup\{\|\phi(b)\|: b\in \Hilm[B], \|b\|\leq 1\}$$
is less than or equal to $1$, we will say that $\phi$ is contractive and completely positive. We will use cp and ccp to abbreviate the terms ``completely positive'' and ``contractive and completely positive'', respectively.
\end{defn}

Note that for $\Cst$\nb-algebras (i.e. Fell bundles over the trivial group $\{e\}$) our definitions of cp and ccp maps are the usual ones.
Also, the restriction of a cp or ccp map of $\Cst(\Hilm[B])$ to $\Hilm[B]$ is a cp or ccp map of $\Hilm[B]$ respectively. One of our goals is to show that every ccp map of $\Hilm[B]$ arises in this way. In what follows we shall write $\phi_e\coloneqq\phi|_{B_e}$ for the restriction of $\phi$ to $\phi_e$. Notice that $\phi_e$ is a cp (resp. ccp) if so is $\phi$.

\begin{lem}\label{lem:adjoint} Let $\phi\colon\Hilm[B]\to\Bound(\Hilm[H])$ be a cp map. Then $\phi$ is self-adjoint in the sense that $\phi(b^*)=\phi(b)^*$ for all $g\in G$ and $b\in B_g$. Moreover, $\phi(b)^*\phi(b)\leq \|\phi_e\|\phi_e(b^*b)$ and 
		$$\|\phi^{\mathfrak{t}}\|=\|\phi\|=\|\phi_e\|=\sup_i \|\phi_e(u_i)\|$$ 
for all $\mathfrak{t}\in G^n$ and $n\in \mathbb{N}$, where $(u_i)$ is any approximate unit for $B_e$.
\begin{proof} For~$g\in G$, take $\mathfrak{t}=(e,g)\in G^2$. The matrix 
\begin{equation*}
M= 
\begin{bmatrix}
 0 & b \\
 b^* & 0 \\
\end{bmatrix}
\end{equation*} is self-adjoint in~$\mbb$. Since $\phi^{\mathfrak{t}}$ is positive, it follows that $\phi^{\mathfrak{t}}(M)^*=\phi^{\mathfrak{t}}(M)$. Thus $\phi(b)^*=\phi(b^*)$. The inequality  $\phi(b)^*\phi(b)\leq \|\phi_e\|\phi_e(b^*b)$ follows as in the proof of \cite{Lance:Hilbert_modules}*{Lemma~5.3} using the positivity of the matrix
\begin{equation*}
\begin{bmatrix}
u_i^2 & u_ib \\
b^*u_i & b^*b \\
\end{bmatrix}=\begin{bmatrix}
u_i & b \\
0 & 0 \\
\end{bmatrix}^*\begin{bmatrix}
u_i & b \\
0 & 0 \\
\end{bmatrix}\in \mbb.
\end{equation*}
From this it follows that $\|\phi\|=\|\phi_e\|$. And the norm equality $\|\phi_e\|=\sup_i \|\phi_e(u_i)\|$ is  \cite{Lance:Hilbert_modules}*{Lemma~5.3(i)} applied to the cp map $\phi_e$.

Finally, if $\phi$ is cp, then so is $\phi^{\mathfrak{t}}$ for every $\mathfrak{t}\in G^n$, $n\in\mathbb N$; this follows from the canonical identification $\mathbb{M}_k(\mbb)=\mathbb{M}_{t^{(k)}}(\Hilm[B])$, where $t^{(k)}=(t,t,\ldots,t)\in G^{kn}$.
Given $a\in B_e$ we write $a^\mathfrak{t}$ for the diagonal matrix in $\mbb$ with constant value $a$ in the diagonal.
Then $B_e\to \mbb, a\mapsto a^\mathfrak{t},$ is an injective \Star homomorphism and, consequently, $\|a\|=\|a^\mathfrak{t}\|$. Moreover, this embedding is nondegenerate, that is, $(u^\mathfrak{t}_i)$ is an approximate unit of $\mbb$ whenever $(u_i)$ is an approximate unit for $B_e$. Therefore, applying again \cite{Lance:Hilbert_modules}*{Lemma~5.3(i)}, now for the cp map $\phi^{\mathfrak{t}}$, we conclude that 
$$\|\phi^{\mathfrak{t}}\|=\sup_i \|\phi^{\mathfrak{t}}(u_i^{\mathfrak{t}})\|=\sup_i\|\phi_e(u_i)^{\mathfrak{t}}\|=\sup_i\|\phi_e(u_i)\|=\|\phi\|.$$
\end{proof}
\end{lem}

\subsection{Stinespring Dilation Theorem for Fell bundles}

We now prove the main theorem of this appendix, which is an extension of Stinespring's Dilation Theorem to Fell bundles.
In fact we will follow Stinespring's original ideas~\cite{Stinespring} in our proof.

\begin{thm}[Stinespring Dilation Theorem for Fell bundles]\label{thm:Stinespring} Let $\Hilm[B]=(B_g)_{g\in G}$ be a Fell bundle over a discrete group and let $\phi\colon \Hilm[B]\to\Bound(\Hilm[H])$ be a cp map. Then there exist a Hilbert space~$\Hilm[K]$, a nondegenerate representation $\pi=\{\pi_g\}_{g\in G}$ of~$\Hilm[B]$ on~$\Hilm[K]$ and a bounded operator $V\colon \Hilm[K]\to\Hilm[H]$ such that $$\phi(b)=V\pi_g(b)V^*$$
for all $g\in G$ and $b\in B_g$.
Moreover, $VV^*$ is the limit of $(\phi(u_i))$ with respect to the strong operator topology of~$\Bound(\Hilm[H])$, where $(u_{i})$ is an approximate identity for~$B_e$.
\begin{proof} Let $\oplus_{\substack{g\in G}}B_g$ be the (algebraic) direct sum of the fibers of $\Hilm[B],$ or equivalently, the set of sections of $\Hilm[B]$ with finite support.
We will write $a_g$ instead of $a(g)$ for all $a\in \oplus_{\substack{g\in G}}B_g.$

We let
$$\Phi\colon((\oplus_{\substack{g\in G}}B_g)\odot\Hilm[H])\times ((\oplus_{\substack{g\in G}}B_g)\odot\Hilm[H])\to\CC$$
be the function given on elementary tensors by $\Phi(a\odot \xi,b\odot \eta)= \sum_{g,h\in G}\braket{\phi(b_g^*a_h)\xi}{\eta}.$
Note that $\Phi$ is linear in the first variable and satisfies $\Phi(x,y)=\overline{\Phi(y,x)}$ (by Lemma \ref{lem:adjoint}). Thus it will be a pre-inner product once we prove that $\Phi(x,x)\geq 0$ for all $x\in (\oplus_{\substack{g\in G}}B_g)\odot\Hilm[H].$

Take $x=\sum_{k=1}^m b^{(k)}\odot \xi_k\in (\oplus_{\substack{g\in G}}B_g)\odot\Hilm[H].$
Let $F=\{t_1,\ldots, t_n\}$ be a finite subset of~$G$ containing the support of~$b_k$ for every~$k\in\{1,\ldots,m\}.$
Set 
$$\mathfrak{t}=\underbrace{(t_1,\ldots, t_n, t_1,\ldots t_n,\ldots, t_1,\ldots,t_n)}_{m \text{ times }}\in G^{nm}$$
and define the element $M\in \mbb$ by
 $$M\coloneqq  \left[\begin{array}{ccccccc}
        (b^{(1)}_{t_1})^*b^{(1)}_{t_1} & \cdots & (b^{(1)}_{t_1})^*b^{(1)}_{t_n} & (b^{(1)}_{t_1})^*b^{(2)}_{t_1} & \cdots & (b^{(1)}_{t_1})^*b^{(k)}_{t_n}\\
        \vdots &    & \vdots   &    \vdots  &  & \vdots \\
        (b^{(1)}_{t_n})^*b^{(1)}_{t_1} & \cdots & (b^{(1)}_{t_n})^*b^{(1)}_{t_n} & (b^{(1)}_{t_n})^*b^{(2)}_{t_1}& \cdots & (b^{(1)}_{t_n})^*b^{(k)}_{t_n})\\
        \vdots & & \vdots & \vdots & & \vdots\\
       (b^{(m)}_{t_1})^*b^{(1)}_{t_1} & \cdots & (b^{(m)}_{t_1})^*b^{(1)}_{t_n} & (b^{(k)}_{t_1})^*b^{(2)}_{t_1}& \cdots &  (b^{(m)}_{t_1})^* b^{(m)}_{t_n}\\
        \vdots &    & \vdots    & \vdots  &  & \vdots \\
       (b^{(m)}_{t_n})^*b^{(1)}_{t_1} & \cdots & (b^{(m)}_{t_n})^*b^{(1)}_{t_n} & (b^{(m)}_{t_n})^*b^{(2)}_{t_1} & \cdots &(b^{(m)}_{t_n})^*b^{(m)}_{t_n}\\
       \end{array}
\right]. $$

Consider the vector~$\xi\in \Hilm[H]^{nm}$ defined by  
$$\xi\coloneqq  (\underbrace{\xi_1,\ldots,\xi_1}_{n \text{ times}},\underbrace{\xi_2,\ldots,\xi_2}_{n \text{ times}}, \ldots, \underbrace{\xi_m,\ldots,\xi_m}_{n \text{ times}} )\in \Hilm[H]^{nm} $$
and note that $\Phi(x,x)=\braket{\phi^{\mathfrak{t}}(M)\xi}{\xi}.$ Since $M\geq 0$ in $\mbb$ by Remark~\ref{rem:compactoperators}, we deduce that $\Phi(x,x)=\braket{\phi^{\mathfrak{t}}(M)\xi}{\xi}\geq 0$, as desired.

We let $\Hilm[K]$ be the Hilbert space built from $(\oplus_{\substack{g\in G}}B_g)\odot\Hilm[H]$ by taking the completion of the quotient space $\big((\oplus_{\substack{g\in G}}B_g)\odot\Hilm[H]\big)/\Hilm[R]$ with respect to the norm induced by the inner product $(b+\Hilm[R],c+\Hilm[R])\mapsto \Phi(b,c)$, where
$$\Hilm[R]\coloneqq\{b\in (\oplus_{\substack{g\in G}}B_g)\odot\Hilm[H]\mid \Phi(b,b)=0\}.$$
The class $a\odot \xi+\Hilm[R]$ of the elementary tensor $a\odot \xi$ will be denoted by $a\otimes \xi.$

We will now construct a representation $\pi=\{\pi_g\}_{g\in G}$ of~$\Hilm[B]=(B_g)_{g\in G}$ on~$\Hilm[K].$
For $h\in G$, $a\in B_h$ and $b\in \oplus_{\substack{g\in G}}B_g $ we define $\pi^0_h(a)b\in \oplus_{\substack{g\in G}}B_g $ by $(\pi^0_h(a)b)_g= a b_{h^{-1}g}.$
Then $\pi^0\coloneqq(\pi^0_g)_{g\in G}\colon \Hilm[B]\to \operatorname{Lin}(\oplus_{\substack{g\in G}}B_g)$ is fibrewise linear and multiplicative, and so is $\pi^0\odot \operatorname{id}_{\Hilm[H]} \coloneqq(\pi^0_g\odot \operatorname{id}_{\Hilm[H]})_{g\in G}\colon \Hilm[B]\to \operatorname{Lin}((\oplus_{\substack{g\in G}}B_g)\odot \Hilm[H]).$
In addition, 
$$\Phi( (\pi^0\odot \operatorname{id}_{\Hilm[H]})(a) x,y )=\Phi(  x,(\pi^0\odot \operatorname{id}_{\Hilm[H]})(a^*)y )$$
for all $a\in \Hilm[B]$ and $x,y\in (\oplus_{\substack{g\in G}}B_g)\odot \Hilm[H].$
We claim that 
\begin{equation}\label{equ:pi0 odot idH is bounded}
 \Phi( (\pi^0\odot \operatorname{id}_{\Hilm[H]})(a) x,(\pi^0\odot \operatorname{id}_{\Hilm[H]})(a)x )\leq \|a\|^2 \Phi(  x,x ).
\end{equation}

Fix $x\in (\oplus_{\substack{g\in G}}B_g)\odot \Hilm[H].$
Then for all $a\in \Hilm[B]$ we have
$$ \Phi( (\pi^0\odot \operatorname{id}_{\Hilm[H]})(a) x,(\pi^0\odot \operatorname{id}_{\Hilm[H]})(a) x ) = \Phi( (\pi^0\odot \operatorname{id}_{\Hilm[H]})(a^*a) x, x ). $$
Thus the functional 
$$ \varphi\colon B_e\to \CC,\ \varphi(a)= \Phi( (\pi^0\odot \operatorname{id}_{\Hilm[H]})(a) x, x )$$
is (completely) positive. By Lemma~\ref{lem:adjoint}  (see also \cite[Theorems 3.3.1, 3.3.2]{Murphy1990book}), for an approximate unit $(u_i)$ of $B_e$ and $a\in \Hilm[B]$ we have
$$ \Phi( (\pi^0\odot \operatorname{id}_{\Hilm[H]})(a) x,(\pi^0\odot \operatorname{id}_{\Hilm[H]})(a) x )
= \varphi(a^*a)\leq \|a\|^2 \lim_i \varphi(u_i).$$
But note that considering the expression $x=\sum_{k=1}^m b^{(k)}\odot \xi_k$ and using the continuity of $\phi$ we obtain
$$ \lim_i \varphi(u_i)
= \lim_i \sum_{j,k=1}^m \sum_{g,h\in G}\braket{ \phi({b^{(j)}_g}^*u_ib^{(k)}_h )\xi_k}{\xi_j}
= \sum_{j,k=1}^m \sum_{g,h\in G} \braket{\phi({b^{(j)}_g}^* b^{(k)}_h )\xi_k}{\xi_j} = \Phi(x,x).$$
Hence \eqref{equ:pi0 odot idH is bounded} follows.

The construction above implies the existence of a \Star representation $\pi\colon \Hilm[B]\to \Bound(\Hilm[K])$ such that  $\pi(a)(b\otimes \xi)= (\pi^0(a)b)\otimes \xi$ for all $a\in \Hilm[B],$ $b\in \oplus_{\substack{g\in G}}B_g$ and $\xi\in \Hilm[H].$

The bilinear function $(\oplus_{\substack{g\in G}}B_g) \times \Hilm[H] \to \Hilm[H],\ (a,\xi)\mapsto \sum_{g\in H} \phi(a_x)\xi,$ induces a linear map $V_0\colon (\oplus_{\substack{g\in G}}B_g)\odot \Hilm[H]\to \Hilm[H]$ such that $V_0(a\odot \xi)=\sum_{g\in G}\phi(a_g)\xi.$
We claim that $V_0$ is bounded with respect to the seminorm defined by $\Phi.$ Indeed, take $x=\sum_{k=1}^m b^{(k)}\odot \xi_k$ and note that for any approximate unit $(u_i)$ of $B_e\subset \oplus_{\substack{g\in G}}B_g$ we have 
 \begin{equation*}\begin{aligned}
 \|V_0(x)\|^2
 =\braket{V_0(x)}{V_0(x)}
    &=\braket{\sum_{\substack{k=1}}^m\big(\sum_g\phi(b_g^{(k)})\big)\xi_k}{\sum_{\substack{k=1}}^m\big(\sum_g\phi(b_g^{(k)})\big)\xi_k}\\
    &=\lim_{\substack{i}}\braket{\sum_{\substack{k=1}}^m\big(\sum_g\phi({u_i}^* b_g^{(k)})\big)\xi_k}{\sum_{\substack{k=1}}^m\big(\sum_g\phi(b_g^{(k)})\big)\xi_k}\\
    &=\lim_{\substack{i}}\Phi(x,u_i \odot V_0(x))\\
    &\leq \sup_{\substack{i}} \|x\|\| u_i\odot V_0(x)\|\\
    &\leq \sup_{\substack{i}}\|x\| \braket{ \phi(u_i^*u_i )V_0(x)}{V_0(x)}^{1/2}\\
    &\leq \|\phi\|\|x\|\|V_0(x)\|.
\end{aligned}
\end{equation*}
This implies $\|V_0(x)\|\leq \|\phi\|\|x\|$ and, consequently, there exists a unique bounded operator $V\colon \Hilm[K]\to \Hilm[H]$ with $\|V\|\leq \|\phi\|$ such that $V(a\otimes \xi)=\sum_{g\in G}\phi(a_g)\xi.$

To establish the equality $\phi(a)=V\pi_g(a)V^*$, we take an approximate identity $(u_{i})$ for~$B_e$.
Observe that for every $\xi\in \Hilm[H]$ and every $x=\sum_{\substack{k=1}}^mb_k\otimes \xi_k \in \Hilm[K]$, we have
\begin{equation*}
\begin{aligned}
\braket{x}{V^*\xi}
  =\braket{Vx}{\xi}
      & =\sum_{\substack{k=1}}^m \sum_{g\in G} \braket{\phi(b^{(k)}_g)\xi_k}{\xi}
         =\lim_i \sum_{\substack{k=1}}^m \sum_{g\in G} \braket{\phi({u_i}^*b^{(k)}_g)\xi_k}{\xi}\\
      &=\lim_{i}\braket{x}{u_{i}\otimes\xi}.
\end{aligned}
\end{equation*}
Since the net $(u_i\otimes \xi)$ is bounded we obtain that $\braket{x}{V^*\xi}=\lim_{\substack{i}}\braket{x}{u_i\otimes \xi}$ for all $x\in \Hilm[K]$ and $\xi\in \Hilm[H].$
Hence, for every $g\in G$, $a\in B_g$ and $\xi,\eta\in\Hilm[H]$, 
\begin{equation*}\begin{aligned}
\braket{V\pi_g(a)V^*\xi}{\eta}
 & =\braket{\pi_g(a)V^*\xi}{V^*\eta}
   = \lim_i \braket{\pi_g(a)V^*\xi}{u_i \otimes \eta}
   = \lim_i \braket{V^*\xi}{a^*u_i \otimes \eta}\\
 & = \braket{V^*\xi}{a^* \otimes \eta}
   = \lim_i \braket{\phi( {u_i}^* a )\xi}{\eta}
   = \braket{\phi(a)\xi}{\eta}.
\end{aligned}
\end{equation*}
 We conclude that $\phi(a)=V\pi_g(a)V^*$ as desired.
It is clear from the construction that $(\pi_e(u_{i}))$ converges to $1_{\Hilm[K]}$ strongly in~$\Bound(\Hilm[K])$. But this means that $\pi$ is nondegenerate and also implies that $VV^*$ is the limit of~$(\phi(u_{i}))$ in the strong operator topology of~$\Bound(\Hilm[H])$. This finishes the proof of the theorem.
\end{proof}
\end{thm}

\begin{rem} The theorem above holds for Fell bundles over locally compact groups and cp maps which are continuous in the weak operator topology.
 This is so because the continuity condition is equivalent to the continuity of the representation~$\pi$ we constructed in the proof above. See~\cite{FellDoran} for further details on Fell bundles over locally compact groups.
\end{rem}

An immediate consequence of Theorem~\ref{thm:Stinespring} is that every cp map of a Fell bundle can be integrated to a cp map of its cross-sectional $\Cst$\nb-algebra.

\begin{cor}
Every cp map $\phi$ of a Fell bundle $\Hilm[B]=(B_g)_{g\in G}$ over a discrete group is the restriction of a unique cp map $\tilde{\phi}$ of $C^*(\Hilm[B])$ with $\|\phi\|=\|\tilde\phi\|$.
\end{cor}
We call $\tilde\phi$ the integrated form of $\phi$.
\begin{proof}Let  $\phi\colon \Hilm[B]\to\Bound(\Hilm[H])$ be a cp map and consider the Hilbert space $\Hilm[K],$ the linear contraction $V\colon \Hilm[K]\to\Hilm[H]$ and the representation $\pi=\{\pi_g\}_{g\in G}\colon \Hilm[B]\to \Bound(\Hilm[K])$ given by Theorem~\ref{thm:Stinespring}. Let $\tilde{\pi}\colon \Cst(\Hilm[B])\to\Bound(\Hilm[K])$ be the \Star homomorphism obtained by the universal property of~$\Cst(\Hilm[B])$. We define a cp map  
\begin{equation*}
\begin{aligned}
\tilde{\phi}\colon  \Cst(\Hilm[B])&\to\Bound(\Hilm[H])\\ 
b&\mapsto V\tilde{\pi}(b)V^*.
\end{aligned}
\end{equation*}
Since $\tilde{\pi}(a)=\pi_g(a)$ for all~$g\in G$ and $a\in B_g$, we also have $\tilde{\phi}(a)=\phi(a)$. Uniqueness follows because $\bigoplus_{\substack{g\in G}}B_g$ is dense in~$\Cst(\Hilm[B])$.
\end{proof}

The next proposition follows as a combination of Theorem~\ref{thm:Stinespring} and the previous corollary. 

\begin{prop}\label{prop:universal} Let $\Hilm[B]=(B_g)_{g\in G}$ and $\Hilm[C]=(C_g)_{g\in G}$ be Fell bundles. Let $\pi=\{\pi_g\}_{g\in G}$ be a faithful representation of~$\Hilm[C]$ on a Hilbert space~$\Hilm[H]$. Suppose that $\phi\colon\Hilm[B]\to\Bound(\Hilm[H])$ is a cp map such that $\phi(B_g)\subseteq \pi_g(C_g)$ for all~$g\in G$. Then there is a cp map $\phi'\colon \Cst(\Hilm[B])\to\Cst(\Hilm[C])$ with $\|\phi'\|=\|\phi\|$ such that $\pi\circ \phi'|_{\Hilm[B]} = \phi.$
\begin{proof} We may regard $\Cst(\Hilm[C])$ as a $\Cst$\nb-algebra of $\Bound(\Hilm[K]),$ for some Hilbert space $\Hilm[K].$
In this way, given $\mathfrak{t}=(t_1,\ldots, t_n)\in G^n$ we view $\mathbb{M}_{\mathfrak{t}}(\Hilm[C])\subset \mathbb{M}_n(\Cst(\Hilm[C]))$ as a $\Cst$\nb-subalgebra of $\Bound(\Hilm[K]^n).$
Thus there exists a unique fibrewise linear map $\kappa\colon \Hilm[B]\to \Bound(\Hilm[K])$ such that for all $g\in G$ and $a\in B_g,$ $\kappa(a)=\pi_g^{-1}(\phi(a)).$
By construction we get $\kappa(\Hilm[B])\subset \Cst(\Hilm[C])$ and $\pi\circ\kappa=\phi,$ implying that   $\pi^\mathfrak{t}\circ {\kappa}^\mathfrak{t} = \phi^\mathfrak{t}$ for every $\mathfrak{t}\in G^n$ and $n\in\NN$.

Note $\pi^\mathfrak{t}\colon \mathbb{M}_{\mathfrak{t}}(\Hilm[C]) \to \Bound(\Hilm[H])$ is an injective \Star homomorphism, thus it is isometric and a matrix $M\in \mbb$ is positive if and only if $\pi^\mathfrak{t}(M)\geq 0.$
Hence, ${\kappa}^\mathfrak{t}$ is contractive and positive if and only if $\pi^\mathfrak{t}\circ {\kappa}^\mathfrak{t}=\phi^\mathfrak{t}$ is so.
Since this last assertion holds by assumption, we deduce that $\kappa$ is cp and $\|\kappa\|=\|\phi\|$.

We know $\kappa\colon \Hilm[B]\to \Bound(\Hilm[K])$ has a unique cp extension with same norm $\phi'\coloneqq\tilde{\kappa}\colon \Cst(\Hilm[B])\to \Bound(\Hilm[K]),$ whose range is contained in $\overline{\kappa(\Hilm[B])}\subset \Cst(\Hilm[C]).$
Thus we may view $\phi'$ as a ccp map from $\Cst(\Hilm[B])$ to $\Cst(\Hilm[C])$ and by construction we get $\phi'|_{B_g}=\kappa|_{B_g}=\pi_g^{-1}\circ \phi|_{B_g}.$
\end{proof}
\end{prop}

\subsection*{Acknowledgements} A. Buss was supported by CNPq and Capes/Humboldt. C. F. Sehnem was supported by CAPES/PrInt grant No. 88887.370650/2019-00 at Universidade Federal de Santa Catarina and by the Marsden Fund of the Royal Society of New Zealand, grant No. 18-VUW-056.


\begin{bibdiv}

  \begin{biblist}


\bib{Abadie:Tensor}{article}{
  author={Abadie, Fernando},
  title={Tensor products of Fell bundles over discrete groups},
  status={eprint},
  note={\arxiv{funct-an/9712006v1}},
  date={1997},
}

\bib{Abadie}{article}{
author = {Abadie, Fernando},
year = {2000},
month = {08},
pages = {},
title = {Enveloping Actions and Takai Duality for Partial Actions},
volume = {197},
journal = {J. Funct. Anal.},
doi = {10.1016/S0022-1236(02)00032-0}
}


\bib{2019arXiv190703803A}{article}{
 author = {{Abadie}, Fernando},
author={ {Buss}, Alcides},
author={ {Ferraro}, Dami{\'a}n},
        title = {Amenability and approximation properties for partial actions and Fell bundles},
      journal = {Bull. Braz. Math. Soc.},
      year = {2021},
doi={10.1007/s00574-021-00255-8},   
}


\bib{Abadieferraro}{article}{
       author = {{Abadie}, Fernando},
author={ {Ferraro}, Dami{\'a}n},
        title = {Equivalence of Fell bundles over groups},
      journal = {J. Oper. Theory},
volume={81},
pages={276--319},
number={2},
year={2019},     
doi={10.7900/jot.2018feb02.2211},     
}



\bib{Anantharaman-Delaroche1987}{article}{
author={Anantharaman-Delaroche, Claire},
year={ 1987},
title={Syst{\`e}mes dynamiques non commutatifs et moyennabilit{\'e}},
journal={Math. Ann.},
pages={297--315},
volume={ 279},
doi={10.1007/BF01461725}}


\bib{anantharamandelaroche2016remarks}{article}{
    title={Some remarks about the weak containment property for groupoids and semigroups},
    author={Anantharaman-Delaroche, Claire},
    year={2016},
note={\arxiv{1604.01724v5}},
    archivePrefix={arXiv},
    primaryClass={math.OA}
}

\bib{Ara-Exel-Katsura:Dynamical_systems}{article}{
  author={Ara, Pere},
  author={Exel, Ruy},
  author={Katsura, Takeshi},
  title={Dynamical systems of type $(m,n)$ and their $\textup{C}^*$\nobreakdash-algebras},
  journal={Ergodic Theory Dynam. Systems},
  volume={33},
  date={2013},
  number={5},
  pages={1291--1325},
  issn={0143-3857},
  doi={10.1017/S0143385712000405},
 review={\MRref{3103084}{}}
}


\bib{Ara-Goodearl:C-algebras_separated_graphs}{article}{
  author={Ara, Pere},
  author={Goodearl, Ken R.},
  title={$C^*$\nobreakdash-algebras of separated graphs},
  journal={J. Funct. Anal.},
  volume={261},
  date={2011},
  number={9},
  pages={2540--2568},
  issn={0022-1236}, 
  doi={10.1016/j.jfa.2011.07.004},
review={\MRref{2826405}{(2012f:46093)}}
}


\bib{blecher2004operator}{book}{
  title={Operator Algebras and Their Modules: An Operator Space Approach},
  author={Blecher, David P.},
author={Le Merdy,  Christian},
volume={30},
  isbn={9780198526599},
  lccn={2005272320},
  series={London Mathematical Society monographs. New series},
  url={https://books.google.com.br/books?id=oBh0EE0HQLQC},
  year={2004},
  publisher={Clarendon, Oxford University Press},
place={Oxford}
}

\bib{Brieskorn1972}{article}{
author={Brieskorn, Egbert},
author={Saito, Kyoji},
year={1972},
title={Artin-Gruppen und Coxeter-Gruppen},
journal={Invent. Math.},
pages={245--271},
volume={17},
number={4},
doi={10.1007/BF01406235},
review={\MRref{0323910}{48\#2263}}
}


\bib{Brown-Ozawa:Approximations}{book}{
  author={Brown, Nathanial P.},
  author={Ozawa, Narutaka},
  title={$C^*$\nobreakdash-algebras and finite-dimensional approximations},
  series={Graduate Studies in Mathematics},
  volume={88},
  publisher={Amer. Math. Soc.},
  place={Providence, RI},
  date={2008},
  pages={xvi+509},
  isbn={978-0-8218-4381-9},
  isbn={0-8218-4381-8},
  review={\MRref{2391387}{2009h:46101}},
}




\bib{Buss-Echter-Willett}{book}{
 author = {Buss, Alcides},
author= {Echterhoff, Siegfried},
author={Willett, Rufus},
title = {Injectivity, crossed products, and amenable group actions},
year = {2020},
  pages = {105--137},
volume= {749},
series=  {Contemp. Math.},
 booktitle = {{$K$}-theory in algebra, analysis and topology},
review = {\MRref{4087636}{}},
doi={10.1090/conm/749/15069},
publisher = {Amer. Math. Soc., Providence, RI}
}



\bib{buss2020amenability}{article}{
    title={Amenability and weak containment for actions of locally compact groups on $\Cst$-algebras},
    author = {Buss, Alcides},
author= {Echterhoff, Siegfried},
author={Willett, Rufus},
    year={2020},
note={\arxiv{2003.03469v5}},
    archivePrefix={arXiv},
    primaryClass={math.OA}
}


\bib{Cohen2002}{article}{
author={Cohen, Arjeh M.},
author={Wales, David B.}
year={2002},
title={Linearity of Artin groups of finite type},
journal={Isr. J. Math.},
pages={101--123},
volume={131},
doi={10.1007/BF02785852},
review={\MRref{1942303}{(2003j:20062)}}}

\bib{crisp_laca_2002}{article}{title={On the Toeplitz algebras of right-angled and finite-type Artin groups}, volume={72}, 
doi={10.1017/S1446788700003876}, 
review={\MRref{1887134}{2003a:46079}},
number={2}, 
journal={J. Aust. Math. Soc.}, 
publisher={Cambridge University Press}, 
author={Crisp, John},
author={Laca, Marcelo}, 
year={2002}, 
pages={223–246}}


\bib{CLEY}{book}{
    author={Cuntz, Joachim},
author ={Echterhoff, Siegfried},
author={Li, Xin},
author={Yu, Guoliang},
series={Oberwolfach seminars},
title={K-Theory for Group $\Cst$\nb-Algebras and Semigroup $\Cst$\nb-Algebras},
  publisher={Birkh\"{a}user/Springer},
place={Cham},
volume={47},
doi={10.1007/978-3-319-59915-1},
year={2017},
     ISSN = {1661-237X},
pages={x--322}
}



\bib{Exel:Circle_actions}{article}{
  author={Exel, Ruy},
  title={Circle actions on $C^*$\nobreakdash-algebras, partial automorphisms, and a generalized Pimsner--Voiculescu exact sequence},
  journal={J. Funct. Anal.},
  volume={122},
  date={1994},
  number={2},
  pages={361--401},
  issn={0022-1236},
  review={\MRref{1276163}{95g:46122}},
  doi={10.1006/jfan.1994.1073},
}


\bib{Exel:Partial_dynamical}{book}{
  author={Exel, Ruy},
  title={Partial dynamical systems, Fell bundles and applications},
  series={Mathematical Surveys and Monographs}
  volume={224},
  date={2017},
  pages={321},
  isbn={978-1-4704-3785-5},
  isbn={978-1-4704-4236-1},
  publisher={Amer. Math. Soc.},
  place={Providence, RI},
}





\bib{FellDoran}{book}{
    AUTHOR = {Fell, James M. G.},
AUTHOR={Doran, Robert S.},
     TITLE = {Representations of {$^*$}-algebras, locally compact groups,
              and {B}anach {$^*$}-algebraic bundles.},
    SERIES = {Pure and Applied Mathematics},
    VOLUME = {125 and 126},
 PUBLISHER = {Academic Press, Boston, MA},
      YEAR = {1988},
}



\bib{GUENTNER2002411}{article}{
title = {Exactness and the Novikov conjecture},
journal = {Topology},
volume = {41},
number = {2},
pages = {411--418},
year = {2002},
issn = {0040-9383},
doi = {0.1016/S0040-9383(00)00036-7},
url = {http://www.sciencedirect.com/science/article/pii/S0040938300000367},
author = {Guentner, Erik },
author={Kaminker, Jerome },
review={\MRref{1876896}{(2003e:46097a)}}}




\bib{Haagerup}{article}{
author={Haagerup, Uffe },
year={1975},
title={The standard form of von Neumann algebras},
journal={Math. Scand.}
  pages={271--283},
volume={37},
doi={10.7146/math.scand.a-11606},
}



\bib{ABCD}{article}{
  author={an Huef, Astrid},
  author={Nucinkis, Brita},
author={Sehnem, Camila F.},
  author={Yang, Dilian},
journal= {J. Funct. Anal.},
title={Nuclearity for semigroup~$\Cst$-algebras},
year={2021},
number = {2},
       doi = {10.1016/j.jfa.2020.108793},
volume = {280},
  review={\MRref{4159272}{}}
}






\bib{Lance:Hilbert_modules}{book}{
  author={Lance, E. {Ch}ristopher},
  title={Hilbert $C^*$\nobreakdash-modules},
  series={London Mathematical Society Lecture Note Series},
  volume={210},
  publisher={Cambridge University Press},
  place={Cambridge},
  date={1995},
  pages={x+130},
  isbn={0-521-47910-X},
  review={\MRref{1325694}{96k:46100}},
  doi={10.1017/CBO9780511526206},
}



\bib{Li:Semigroup_amenability}{article}{
  author={Li, Xin},
  title={Semigroup $\textup C^*$\nobreakdash-algebras and amenability of semigroups},
  journal={J. Funct. Anal.},
  volume={262},
  date={2012},
  number={10},
  pages={4302--4340},
  issn={0022-1236},
  review={\MRref{2900468}{}},
  doi={10.1016/j.jfa.2012.02.020},
}

\bib{LI2013626}{article}{
title = {Nuclearity of semigroup $\Cst$\nb-algebras and the connection to amenability},
journal = {Adv. Math.},
volume = {244},
pages = {626-- 662},
year = {2013},
issn = {0001-8708},
doi = {10.1016/j.aim.2013.05.016},
url = {http://www.sciencedirect.com/science/article/pii/S0001870813001916},
  author={Li, Xin}}


\bib{LOS}{article}{
author = {Li, Xin},
author={Omland, Tron},
author={Spielberg, Jack},
year = {2021},
volume={381},
numer={3},
journal = {Comm. Math. Phys.},
pages = {1263--1308},
title = {$\Cst$-algebras of right LCM one-relator monoids and Artin-Tits monoids of finite type},
doi = {10.1007/s00220-020-03758-5},
}



\bib{Lolk:Exchange}{article}{
  author={Lolk, Matias},
title={Exchange rings and real rank zero C*-algebras associated with finitely separated graphs},
year={2017},
note={\arxiv{1705.04494v1}}}

\bib{Matsumura}{article}{
author = {Matsumura, Masayoshi},
year = {2012},
month = {04},
pages = {},
title = {A characterization of amenability of group actions on $\Cst$\nb-algebras},
volume = {72},
journal = {J. Oper. Theory},
doi = {10.7900/jot.2012sep07.1958}
}



\bib{Mcclanahan1995KTheoryFP}{article}{
title = {K-Theory for Partial Crossed Products by Discrete Groups},
journal = {J. Funct. Anal.},
volume = {130},
number = {1},
pages = {77--117},
year = {1995},
issn = {0022-1236},
doi = {10.1006/jfan.1995.1064},
review={\MRref{1331978}{(96i:46083)}},
  author={McClanahan, Kevin Paul}
}

\bib{Murphy1990book}{book}{
title={{C}*-algebras and operator theory},
author={Murphy, Gerard J},
volume={288},
year={1990},
publisher={Academic press Boston}
}

\bib{Nica:Wiener--hopf_operators}{article}{
  ISSN = {0379-4024},
 URL = {http://www.jstor.org/stable/24715075},
 author = {Nica, A.},
 journal = {J. Oper. Theory},
 number = {1},
 pages = {17--52},
 publisher = {Theta Foundation},
 title = {$C^*$\nobreakdash-algebras generated by isometries and Wiener--Hopf operators},
 review={\MRref{1241114}{46L35 (47B35 47C10)}},
 volume = {27},
 year = {1992}
}


\bib{Ozawa}{article}{
author = {Ozawa, Narutaka},
year = {2000},
month = {02},
pages = {},
title = {Amenable actions and exactness for discrete groups},
volume = {330},
journal = {C. R. Acad. Sci. Paris Sér. I Math.},
review={\MRref{1763912}{(2001g:22007)}},
doi = {10.1016/S0764-4442(00)00248-2}
}


\bib{Raeburn-Williams:Morita_equivalence}{book}{
  author={Raeburn, Iain},
  author={Williams, Dana P.},
  title={Morita equivalence and continuous-trace $C^*$\nobreakdash-algebras},
  series={Mathematical Surveys and Monographs},
  volume={60},
  publisher={Amer. Math. Soc.},
  place={Providence, RI},
  date={1998},
  pages={xiv+327},
  isbn={0-8218-0860-5},
  review={\MRref{1634408}{2000c:46108}},
  doi={10.1090/surv/060},
}


\bib{Sehnem}{article}{
author = {Sehnem, Camila F.},
year = {2019},
pages = {558--593},
title = {On $\Cst$\nb-algebras associated to product systems},
volume = {277},
number={2},
journal = {J. Funct. Anal.},
doi = {10.1016/j.jfa.2018.10.012},
review={\MRref{3952163}{}}
}

\bib{Stinespring}{article}{
author={Stinespring, W. Forrest},
journal={Proc. Amer. Math. Soc.},
volume={6},
title={Positive functions on $\Cst$\nb-algebras},
year={1955},
pages={211--216},
  review={\MRref{69403 }{46.0X}},
  doi={10.1090/S0002-9939-1955-0069403-4},
}
\end{biblist}
\end{bibdiv}
\end{document}